\documentclass[12pt,reqno]{amsart}
\usepackage[T1]{fontenc}
\usepackage[utf8]{inputenc}
\usepackage{amssymb,latexsym,amsthm,amsfonts,amscd,pgf, enumerate,ragged2e}
\usepackage{amsmath}
\usepackage{pstricks-add}
\usepackage{graphics,pstricks}
\usepackage{pgf,tikz}
\usepackage{mathrsfs}
\usepackage{fullpage}
\usepackage{microtype,textcomp,csquotes, lmodern}
\usetikzlibrary{arrows}
\usepackage[dvipsnames]{xcolor}
\definecolor{darkblue}{rgb}{0,0,0.7}
\definecolor{darkred}{rgb}{0.7,0,0}
\usepackage[colorlinks=true, linkcolor=darkred, citecolor=darkblue, urlcolor=blue, pagebackref=true, breaklinks=true]{hyperref}
\usepackage[capitalise]{cleveref}
\usepackage{enumitem}
\usepackage[margin=0.9in]{geometry}

\newtheorem{proposition}{Proposition}[section]
\newtheorem{lemma}[proposition]{Lemma}
\newtheorem{theorem}[proposition]{Theorem}
\newtheorem{corollary}[proposition]{Corollary}
\newtheorem{question}[proposition]{Question}
\newtheorem{conjecture}[proposition]{Conjecture}
\theoremstyle{definition}
\newtheorem{remark}[proposition]{Remark}

\newtheorem{example}[proposition]{Example}
\newtheorem{definition}[proposition]{Definition}

\newtheorem{innercustomthm}{Theorem}
\newenvironment{customthm}[1]
  {\renewcommand\theinnercustomthm{#1}\innercustomthm\itshape}
  {\endinnercustomthm}

\newcommand{\lk}{{\mathrm{lk}}}
\newcommand{\del}{{\mathrm{del}}}
\newcommand{\Con}{{\mathrm{Con}}}
\newcommand{\Supp}{\mathrm{Supp}}

\usepackage{tikz,float}
\usepackage{caption}
\usepackage{subcaption}
\usepackage{comment}
\usetikzlibrary{positioning, shapes, fit, arrows}
\usepackage{tkz-graph}
\tikzstyle{vertex}=[circle, draw, inner sep=0pt, minimum size=6pt]

\def\F{\mathcal{F}}
\def\H{\mathcal{H}}

\def\E{\mathcal{E}}
\def\N{\mathcal{N}}

\def\D{\mathcal{D}}

\def\link{\mathrm{link}}
\def\del{\mathrm{del}}
\def\Ind{\mathrm{Ind}}
\def\l{\langle}
\def\r{\rangle}
\def\x{\mathbf x}

\def\SMS{\mathcal{SMS}}
\def\MS{\mathcal{MS}}
\def\SrAG{\Sigma_r(A,G)}
\def\Indr{\mathrm{Ind}_r(G)}

\title{The complex of $r$-co-connected subgraphs, chordality and Fr\"oberg's theorem}

\thanks{Last updated: \today}

\author{Priyavrat Deshpande}
\address{Chennai Mathematical Institute, India}
\email{pdeshpande@cmi.ac.in}

\author{Amit Roy}
\address{Chennai Mathematical Institute, India}
\email{amitiisermohali493@gmail.com}

\author{Rutuja Sawant}
\address{Chennai Mathematical Institute, India}
\email{rutuja@cmi.ac.in}
\address{}
\email{}
\thanks{}

\begin{document}
	
	\begin{abstract}

We introduce a new family of pure simplicial complexes, called the $r$-co-connected complex of $G$ with respect to $A$, $\Sigma_r(A,G)$, where $r\geq 1$ is a natural number, $G$ is a simple graph, and $A$ is a subset of vertices.  
Interestingly, when $A$ is empty, this complex is precisely the Alexander dual of the $r$-independence complex of $G$. 
We focus on uncovering the relationship between the topological and combinatorial properties of the complex and the algebraic and homological properties of the Stanley-Reisner ideal of the dual complex. 
A key ingredient in this proof is an auxilary complex that encodes local connectivity. 
First, we prove that $\Sigma_r(A,G)$ is vertex decomposable whenever the induced subgraph $G[A]$ is connected and nonempty, yielding a versatile deletion-link calculus for higher independence via Alexander duality. 
Furthermore, when $A=\emptyset$ and $r \ge 2$, we establish that for several significant classes of graphs - including chordal, co-chordal, cographs, cycles, complements of cycles, and certain grid graphs - the properties of vertex decomposability, shellability, and Cohen-Macaulayness are equivalent and precisely characterized by the co-chordality of the associated clutter $\mathrm{Con}_r(G)$. These results extend Fr\"oberg's theorem to the setting of $r$-connected ideals for these graph classes and motivate a conjecture concerning the linear resolution property of $r$-connected ideals in general. We also construct examples separating shellability from vertex decomposability.
\end{abstract}

\maketitle
	
\section{Introduction}\label{introduction}

The interplay between topological combinatorics and combinatorial commutative algebra, mediated by the Stanley-Reisner correspondence, has yielded profound connections between the structural properties of graphs and simplicial complexes and the algebraic invariants of monomial ideals. A classical result in this direction is Fr\"oberg’s theorem~\cite{OSR}, which characterizes the chordality of a finite simple graph in terms of the linear resolution property of its associated quadratic square-free monomial ideal.
	
Let $G$ be a finite simple graph on the vertex set $V(G)=\{x_1,\ldots,x_n\}$ with circuit set $E(G)$. A subset $W\subseteq V(G)$ is called an \emph{independent set} if $W$ contains no circuit of $G$. The collection of all independent sets forms the \emph{independence complex} $\mathrm{Ind}(G)$ of $G$, and the associated Stanley--Reisner ideal $I(G)=\langle x_i x_j \mid \{x_i,x_j\}\in E(G)\rangle$ in the polynomial ring $R=\mathbb K[x_1,\ldots,x_n]$ is known as the \emph{circuit ideal} (or edge ideal) of $G$.  
Fr\"oberg’s theorem can then be stated as
	\[
	I(G) \text{ has a linear resolution } \Longleftrightarrow G \text{ is co-chordal},
	\]
where a graph is said to be \emph{co-chordal} if its complement is chordal. In particular, for quadratic square-free monomial ideals, the property of having a linear resolution depends only on the combinatorial structure of the associated graph and not on the characteristic of the base field.
This result was later extended by Eagon and Reiner \cite[Proposition 8]{EagonReiner1998} who showed that several algebraic and topological properties are equivalent for graphs, as summarized below (recall that for a simplicial complex $X$, its Alexander dual, denoted by $X^{\vee}$, is the simplicial complex whose faces are the complement of non-faces of $X$).
\begin{theorem}
    The following are equivalent for a graph $G$:
    \begin{enumerate}
        \item $I(G)$ has a linear resolution,
        \item $(\Ind(G))^{\vee}$ is vertex decomposable,
        \item $(\Ind(G))^{\vee}$ is shellable,
        \item $(\Ind(G))^{\vee}$ is Cohen-Macaulay,
        \item $G$ is co-chordal.
    \end{enumerate}
\end{theorem}

The situation changes drastically in higher degrees. 
A classical example due to Reisner \cite{reisner76}, arising from the minimal triangulation of the projective plane, shows that the minimal free resolution of the corresponding degree $3$ square-free monomial ideal does depend on the characteristic of the field $\mathbb{K}$.
More recently, for $1\le d\le n-3$, Ficarra and Moradi~\cite{CEI} constructed degree-$d$ square-free monomial ideals exhibiting the same dependence. 
Consequently, a purely combinatorial characterization of all square-free monomial ideals with linear resolution-analogous to Fr\"oberg’s theorem-cannot be achieved in general.
	
A more tractable problem is to characterize combinatorially those square-free monomial ideals that have linear minimal free resolutions over every field. 
For such higher degree ideals, hypergraphs or clutters are convenient combinatorial structures, as they generalize graphs. 

In this context, several approaches have been proposed by extending the notion of chordality from graphs to clutters (i.e., hypergraphs). 
Among these, the construction of \emph{chordal clutters} by Bigdeli et~al. \cite{SBCC} is particularly notable, as it subsumes previous notions of chordality introduced by Woodroofe~\cite{CSCM}, Emtander~\cite{Emtander2008}, and Van Tuyl and Villarreal \cite{tuyl_vill08}. 
A $r$-uniform clutter $\mathcal{C}$ is called chordal if it admits a simplicial sequence of maximal $(r-1)$-subcircuits whose iterative deletion exhausts $\mathcal{C}$, generalizing perfect elimination in chordal graphs, and for the complement of such clutters the edge (or circuit) ideal has a linear resolution over every field.
However, a counterexample attributed to Eric Babson \cite[Example 2.4]{CCVD} demonstrated that there is a clutter which is not chordal, but the associated circuit ideal has a linear resolution. 
In the same article, it was shown that if the Stanley-Reisner complex associated with a clutter is vertex decomposable, then the clutter is chordal \cite[Theorem 3.10]{CCVD}. 
These developments naturally lead to the following question.
	\begin{question}
		Can Fr\"oberg’s theorem be expected to hold for specific classes of clutters?
	\end{question}
	
The motivation of our investigation was to extend the results of \cite{FVDH}, in order to obtain a better characterization of the Stanley-Reisner ideal of the $r$-independence complexes. The study of $r$-independence complexes can be traced to the work of Paolini and Salvetti \cite{WSHAG}. 
They are studied in the combinatorial point of view in \cite{CHVD,FVDH,DDIG,HICH,RoyPatra2025}. More recently, the algebraic properties of the corresponding edge (circuit) ideals, referred to as \emph{$r$-connected ideals}, were studied in~\cite{LQCG,2025connectedideals}.

The linear resolution property of edge (circuit) ideals of clutters is intimately connected to the Cohen-Macaulayness of the Alexander dual of the corresponding independence complexes. By the theorem of Eagon and Reiner~\cite{EagonReiner1998}, the edge (circuit) ideal of a clutter has a linear resolution if and only if the Alexander dual of its Stanley-Reisner complex is Cohen-Macaulay. Hence, to investigate the linear resolution property of the edge (circuit) ideal of $r$-connected clutters, it suffices to study the Cohen-Macaulay property of the Alexander dual of their independence complexes. Motivated by this connection, we introduce a new family of simplicial complexes associated to a graph $G$, defined for each subset $A \subseteq V(G)$ by
	\[
	\Sigma_r(A,G) = \big\langle F^c \setminus A : F\subseteq V(G),\; |F|=r,\; F\cap A=\emptyset,\; G[F\cup A]\text{ is connected} \big\rangle.
	\]
	We call $\Sigma_r(A,G)$ the \emph{$r$-co-connected complex} of $G$ with respect to $A$. In particular, $\Sigma_r(\emptyset,G)$, denoted simply by $\Sigma_r(G)$, coincides with the Alexander dual of $\mathrm{Ind}_r(G)$, that is, $\Sigma_r(G)=(\mathrm{Ind}_r(G))^{\vee}$. Our first main result in this context is the following.
    \begin{customthm}{\ref{Vertex decoposable}}
Let $G$ be a graph, and let $A \subseteq V(G)$ be a non-empty set such that $G[A]$ is connected. Then, for each $r \geq 1$, the simplicial complex $\Sigma_r(A,G)$ is vertex decomposable.
    \end{customthm}
	
	In the special case $A=\emptyset$, the situation becomes considerably more intricate. Indeed, there exist graphs $G$ for which the complex $(\mathrm{Ind}_r(G))^{\vee}$ fails to be vertex decomposable. Nevertheless, in \Cref{V.D of Sigma_r(G)}, we establish a necessary and sufficient condition for the vertex decomposability of $\Sigma_r(G)$. Furthermore, for $ r\ge 2$, we prove that for several important classes of graphs, the following properties are equivalent: vertex decomposability, shellability, and the Cohen-Macaulay property of $\Sigma_r(G)$. In addition, these conditions are shown to be equivalent to the co-chordality of the associated $r$-connected clutter $\mathrm{Con}_r(G)$. The classes of graphs in this context are the following:

	\begin{itemize}
		\item[$\bullet$] chordal graphs (\Cref{V.D of Sigma_r(G) if G is chordal}),
		
		\item[$\bullet$] co-chordal graphs (\Cref{V.D of Sigma_r(G) if G is co-chordal}),

        \item[$\bullet$] cographs (\Cref{V.D of Sigma_r(G) if G is cograph}),
		
		\item[$\bullet$] cycle graphs $C_n$ (\Cref{cycle vd theorem}),
		
		\item[$\bullet$]  cycle-complement graphs (\Cref{cycle complement}),
		
		\item[$\bullet$] ladder graphs $P_n\times P_2$ (\Cref{ladder graph}),
		
		\item[$\bullet$] grid graphs $P_n\times P_3$ (\Cref{3ladder graph}).
	\end{itemize}
	
The above results can be viewed as an extension of Fr\"oberg’s theorem to the setting of $r$-connected ideals of the aforementioned classes of graphs. 
Note that, in the case of co-chordal graphs, complements of cycles, and cographs, the simplicial complexes $\Sigma_r(G)$ are already vertex decomposable for every $r \ge 2$. Furthermore, the result in \Cref{V.D of Sigma_r(G) if G is chordal} provides an affirmative answer to the first part of \cite[Question 7.5]{LQCG}. 

Two related complexes to $\Sigma_r(G)$ are the \emph{$r$-cut complex} $\Delta_r(G)$ and the \emph{$r$-path-free complex} $\Gamma_r(G)$. 
It is worth mentioning that $\Sigma_r(G) = \Delta_r(\overline{G}) = \Gamma_r(G)^{\vee}$ for $r \le 3$, where $\overline{G}$ denotes the complement of $G$. 
Hence, our results for $\Sigma_r(G)$ naturally relate to some of the known results on $\Delta_r(\overline{G})$ and $\Gamma_r(G)^{\vee}$ in \cite{LQCG, TCG1, TCG2, 2025connectedideals} for the case $r = 3$.
	
As mentioned earlier, the complex $\Sigma_r(G)$ is not always vertex decomposable. In \Cref{gap-free example}, we present an example of a gap-free graph $G$ for which $\mathrm{Con}_3(G)$ is co-chordal and $\Sigma_3(G)$ is shellable, yet $\Sigma_3(G)$ fails to be vertex decomposable. Furthermore, in \Cref{shell but not VD}, for each $r\ge 3$, we construct an example of an $r$-gap-free unicyclic graph $G$ such that $\Sigma_r(G)$ is shellable but not vertex decomposable. These observations motivate us to propose the following conjecture.
		\begin{conjecture}\label{conj1}
		Let $G$ be a finite simple graph with the $r$-co-connected complex $\Sigma_r(G)$. Then $\Sigma_r(G)$ is Cohen-Macaulay if and only if $\mathrm{Con}_r(G)$ is co-chordal.
	\end{conjecture}
Note that, if this conjecture holds true, it would yield a complete combinatorial classification of graphs whose $r$-connected ideals possess linear resolutions, and in particular, the linear resolution property of $r$-connected ideals will not depend on the characteristic of the base field. Our results presented in \Cref{V.D of Sigma_r(G) if G is chordal}, \Cref{V.D of Sigma_r(G) if G is co-chordal}, \Cref{V.D of Sigma_r(G) if G is cograph}, \Cref{cycle vd theorem}, \Cref{cycle complement}, \Cref{ladder graph}, and \Cref{3ladder graph} provide substantial evidence in support of this conjecture.
At present, no example is known of a graph $G$ for which $\Sigma_r(G)$ is Cohen–Macaulay while $\operatorname{Con}_r(G)$ fails to be co-chordal.

It is natural to ask whether known counterexamples to Fröberg-type characterizations for general clutters, such as Babson’s example (\cite[Example 2.4]{CCVD}), could occur in our setting.
Babson’s construction uses a uniform clutter whose circuits cannot be realized as connected induced subgraphs of a graph.
By contrast, the clutters we study are of the special form $\operatorname{Con}_r(G)$, where each circuit is an $r$-vertex set inducing a connected subgraph of $G$.
This graph-based structure imposes strong hereditary and intersection constraints on the circuits, making the local configurations underlying Babson’s example impossible in $\operatorname{Con}_r(G)$.
Thus, existing counterexamples for arbitrary clutters do not directly challenge our conjecture, which is stated purely in graph-theoretic terms.
Taken together, our results suggest that graph connectivity severely restricts the pathological behavior observed in general higher-degree square-free monomial ideals.

\begin{remark}
Throughout this paper, clutters arise exclusively as auxiliary objects associated to graphs, namely $\Con_r(G)$. 
We do not attempt to study arbitrary clutters. 
Rather, the role of $\Con_r(G)$ is to encode graph-theoretic connectivity in a form amenable to existing results on chordal clutters and linear resolutions. 
All main results are therefore statements about graphs and their associated simplicial complexes $\Sigma_r(G)$.     
\end{remark}

The paper is organized as follows. In \Cref{preliminaries}, we gather the necessary definitions and background material on clutters, simplicial complexes, higher independence complexes, and the Stanley–Reisner ideals associated with these complexes. 
In \Cref{sec-3}, we introduce the notion of the \emph{$r$-co-connected complex} $\Sigma_r(G)$ and its natural generalization $\Sigma_r(A,G)$, and examine their topological properties. 
A key technical ingredient is the introduction and analysis of the auxiliary complexes $\Sigma_r(A, G)$, which allow us to carry out inductive arguments on vertex decomposability and chordality.
In \Cref{Section 4 Froberg}, we establish analogues of Fr\"oberg’s theorem for $r$-connected ideals corresponding to various classes of graphs. 
We also obtain a concrete decision procedure for vertex decomposability of $\Sigma_r(G)$, clarifying the computational content of the theory.
Finally, \Cref{con-rem}, presents several examples demonstrating that certain results established in \Cref{Section 4 Froberg} do not necessarily hold for arbitrary graphs.
The paper concludes with a discussion of open problems and directions for future research.

	\section{Preliminaries}\label{preliminaries}
	
	In this section, we recall some preliminary notions related to simplicial complexes and clutters (hypergraphs) that will be used throughout the paper.
Moreover, we consistently use the parameter $r$ when citing results stated for $d$-uniform clutters, we implicitly take $d=r$.

    \subsection{Simplicial complex} 
	An {\it (abstract) simplicial complex} $\Delta$ on the vertex set $V(\Delta)=\{x_1,\ldots,x_n\}$ is defined as a collection of subsets of $V(\Delta)$ such that, given $U\in \Delta$ and $Z\subseteq U$, then $Z\in \Delta$ holds. 
	The elements of $\Delta$ are referred to as the {\it faces} of $\Delta$, while the maximal faces are called the {\it facets}. 
	The notation $\mathcal F(\Delta)$ is used for the set of facets of $\Delta$. 
	Let $F_1,\ldots, F_t$ be all the facets of $\Delta$, we say that $\Delta$ is generated by $F_1,\ldots, F_t$, written as $\Delta=\langle F_1,\ldots, F_t \rangle$. 
	A simplicial complex $\Delta=\l F_1,\ldots,F_t\r$ is {\it pure} if $|F_i|=|F_j|$, for every $i\neq j$. 
	Concerning $F\in\Delta$, the {\it dimension} of $F$ is defined as $\dim F=|F|-1$, and in the case of the simplicial complex $\Delta$, we define $\dim\Delta=\max_{F\in \Delta}\dim F$. 
	If $F$ is a face of $\Delta$ such that $\dim(F)=\dim(\Delta)-1$, then $F$ is a {\it co-dimension one face}. 
	Therefore, a {\it $d$-simplex} is a $d$-dimensional simplicial complex containing exactly one facet of dimension $d$.
	If $\Delta$ is a simplex on the vertex set $A$, it may sometimes be denoted by $\Delta_{A}$. 
	Assume $\Delta_1$ and $\Delta_2$ are two simplicial complexes on disjoint vertex sets $V(\Delta_1)$ and $V(\Delta_2)$. 
	The {\it join} of $\Delta_1$ and $\Delta_2$, expressed as $\Delta_1*\Delta_2$, constitutes a simplicial complex defined by the vertex set and the set of facets as: $V(\Delta_1*\Delta_2)=V(\Delta_1)\sqcup V(\Delta_2)$, and $\F(\Delta_1*\Delta_2)=\{F_1\sqcup F_2\mid F_1\in\F(\Delta_1),F_2\in\F(\Delta_2)\}$, respectively.
	
	Let $F$ be a face of a simplicial complex $\Delta$. The {\it link} of $F$ is the simplicial complex \[\mathrm{lk}_{\Delta}(F)=\{H \in \Delta|\,H \cap F =\emptyset \hspace{2mm}\text{and}\hspace{2mm} H \cup F \in \Delta \},\]
	and the {\it deletion} of $F$ is the simplicial complex \[\mathrm{del}_{\Delta}(F)=\{ H \in \Delta|\,H \cap F =\emptyset \}.\] 
	In case $F=\{x_i\}$, for brevity, we denote $\mathrm{lk}_{\Delta}(\{x_i\})$ and $\mathrm{del}_{\Delta}(\{x_i\})$ as $\mathrm{lk}(x_i)$ and $\mathrm{del}(x_i)$, respectively. 
	\begin{definition}
		A simplicial complex $\Delta$ is vertex decomposable if $\Delta$ is a simplex, or a vertex $x$ exists in $\Delta$ making $\mathrm{lk}_{\Delta}(x)$ and $\mathrm{del}_{\Delta}(x)$ vertex decomposable, and every facet of $\mathrm{del}_{\Delta}(x)$ is also a facet of $\Delta$.
	\end{definition}
    The {\it empty complex}, denoted by $\{\emptyset\}$, is the simplicial complex consisting only of the empty set as its face; it is regarded as the unique $(-1)$-dimensional simplicial complex.
The {\it void complex}, denoted by $\emptyset$, is the collection with no faces at all and is typically considered as a complex on the empty vertex set.
	Note that the empty and void complexes are considered vertex decomposable. Two more definitions that we need are:
	\begin{definition}
		A simplicial complex $\Delta$ is shellable if its facets $F_1,\ldots,F_s$ can be ordered so that for each $2\le i\le s$, the subcomplex $\l F_i\r\cap \l F_1,\ldots,F_{i-1}\r$ is pure of dimension $\mathrm{dim}(F_i)-1$. Such an ordering of the facets is called a shelling order.
	\end{definition}
	
	\begin{definition}
		A simplicial complex $\Delta$ is Cohen-Macaulay over a field $\mathbb K$ if $\widetilde{H}_i(\mathrm{lk}_{\Delta}(F);\mathbb{K})=0$ holds for all $F\in\Delta$ and $i < \dim \mathrm{lk}_{\Delta} (F)$, where $\widetilde{H}_i(-;\mathbb{K})$ is $i^{\text{th}}$ reduced simplicial homology group with coefficient in $\mathbb K$.   
	\end{definition}
	We drop the reference to $\mathbb{K}$ if the homology is computed with coefficients in $\mathbb{Z}$. 
	For a pure simplicial complex $\Delta$, it is a well-known fact (see, for instance, \cite[Proposition 6.3.32]{RHVBook})
	\begin{align}\label{VdShCm}
		\text{Vertex decomposable}\Rightarrow \text{Shellable}\Rightarrow \text{Cohen-Macaulay}.
	\end{align}

    In the sequel, we discuss the Stanley-Reisner ideal of the $r$-independence complex of a graph and the associated clutter, where the notion of Alexander duality plays an important role.
    \begin{definition}
        Let $\Delta$ be a simplicial complex on the vertex set $V(\Delta)$. Then the Alexander dual of $\Delta$, denoted by $\Delta^{\vee}$, is the simplicial complex with the set of faces $\{V(\Delta)\setminus F:F\notin\Delta\}$.
    \end{definition}
   	

	
	
	\subsection{{Clutters}} 
Though the results in this paper are expressed in terms graph properties we use the generalization of the chordality notion as our main tool. 
Hence, we quickly recall relevant background on clutter. 
We include this material solely to record known results used later; no results about arbitrary clutters are claimed

A {\it clutter} $\H$ is defined as a pair $(V(\H), E(\H))$, consisting of a vertex set $V(\H)$ and a circuit set $E(\H)$, wherein the circuit set is composed of a collection of subsets from $V(\H)$, ensuring that no two elements of these subsets are such that one is contained within another. 
If for each $\E\in E(\H)$, $|\E|=d$, then we define $\H$ as a {\it $d$-uniform clutter}. 
Notably, the $2$-uniform clutters are identified as the finite simple graphs. 
A subset $W\subseteq V(\H)$ is designated as a {\it clique} within a $d$-uniform clutter $\H$ when every subset of $W$ with cardinality $d$ belongs to $E(\H)$. In the context of a $d$-uniform clutter $\H$, the {\it complement} $\H^c$ of $\H$ corresponds to a clutter with the vertex set $V(\H)$ and circuit set $\{\E\subseteq V(\H)\mid |\E|=d\text{ and }\E\notin E(\H)\}$. 
The {\it induced subclutter} $\H[A]$ of a clutter $\H$ on set $A\subseteq V(\H)$ is a clutter with vertex set $A$ and circuit set $\{\E\in E(\H)\mid \E\subseteq A\}$. 
Furthermore, $\H\setminus A$ denotes the induced subclutter $\H[V(\H)\setminus A]$. If $A=\{a\}$ for some $a\in V(\H)$, then $\H\setminus\{a\}$ is simply denoted by $\H\setminus a$. 
For $B\subseteq V(\H)$, the {\it open neighborhood} of $B$, signified by $\N_{\H}(B)$, comprises the set $\{x\in V(\H)\setminus B\mid x\in \E\subseteq  \{x\}\cup B \text{ for some }\E\in E(\H)\}$. 
If $b\in \N_{\H}(B)$, then we say $b$ is a {\it neighbor} of $B$. 
The collection $\N_{\H}(B)\cup B$ is referred to as the {\it closed neighborhood} of $B$, represented by $\N_{\H}[B]$.

Let $\H$ be a clutter with vertex set and circuit set $V(\H)$ and $E(\H)$, respectively. A set $W\subseteq V(\H)$ is said to be an {\it independent set} if for each $\E\in E(\H)$, $\E\cap W=\emptyset$. The collection of all independent sets of $\H$ forms the {\it independence complex} of $\H$, denoted by $\mathrm{Ind}(\H)$.
A {\it matching} $M$ in a clutter $\H$ is characterized by a collection of circuits where $\E\cap \E'=\emptyset$ for any two distinct $\E,\E'\in M$. 
A matching $M$ constitutes an {\it induced matching} of $G$ if $E(\H[V(M)])=M$, wherein $V(M)$ denotes the set of vertices residing in the circuits of $M$. 
Lastly, the matching number $\nu(\H)$ and the induced matching number $\gamma(\H)$ are defined as follows:
\begin{align*}
		\nu(\H)&=\max\{|M|: M\text{ is a matching of }\H\}\\
		\gamma(\H)&=\max\{|M|: M\text{ is an induced matching of }\H\}.	
\end{align*}

Consider $G$, a finite, simple graph defined by vertex set $V(G)$ and circuit set $E(G)$. For any two vertices $a,b\in V(G)$, an {\it induced path of length $t$} connecting $a$ and $b$ is constituted by a collection of vertices $P_t:a=a_1,a_2\ldots,a_{t-1},a_t=b$, such that $\{a_i,a_{i+1}\}$ are the only circuits. 
We simply call \(P_t\) a \emph{path} between \(a\) and \(b\) if $\{a_i,a_{i+1}\}$ are circuits but may not be the only circuits among all the $a_i$.
A {\it cycle} $C_n$ of length $n$ consists of a vertex set $V(C_n)=\{y_1,\ldots,y_n\}$ and a circuit set $E(C_n)=\{\{y_1,y_n\},\{y_i,y_{i+1}\}\mid 1\le i\le n-1\}$. 
A graph $G$ is said to be a \textit{forest} if it does not contain any cycle. 
A connected forest is simply called a \textit{tree}. A {\it spanning tree} of a graph $G$ is a subgraph $H$ of $G$ such that $H$ is a tree and $V(H)=V(G)$. 
A graph $G$ is said to be a \textit{chordal graph} if there is no induced $C_n$ in $G$ for $n\geq 4$. 
A vertex $x\in V(G)$ is said to be a \textit{simplicial vertex} of $G$ if $N_{G}[x]$ is a clique in $G$ (i.e., any two vertices are connected by a circuit). 
Moreover, $x$ is said to be a {\it leaf} vertex of $G$ if $x$ is a simplicial vertex and $|N_G[x]|=2$. 
Two circuits $e$ and $e'$ of a graph $G$ are said to form a \textit{gap} in $G$ if $\{e,e'\}$ forms an induced matching in $G$. A graph $G$ is said to be a {\it gap-free graph} if no two circuits of $G$ form a gap.
	
The notion of a chordal graph is extended to the realm of clutters by several researchers. Recently, Bigdeli et. al. \cite{SBCC} defined a new notion of chordal clutters that is more general in the sense that it absorbs various previous notions of chordality. Let $\H$ be a $d$-uniform clutter. For $e\subseteq V(\H)$, the {\it deletion} of $e$, denoted by $\H - e$, is the clutter on the vertex set $V(\H)$ with circuit set $\{\E\in E(\H)\mid e\nsubseteq \E\}$. A subset $e\subseteq V(\H)$ of cardinality $d-1$ is said to be a {\it maximal subcircuit} if $\N_{\H}(e)\neq\emptyset$. The collection of all maximal subcircuit of $\H$ is denoted by $\MS(\H)$. A mximal subcircuit $e$ of $\H$ is called a {\it simplicial maximal subcircuit} of $\H$ if $\N_{\H}[e]$ forms a clique in $\H$. Analogously, the collection of all simplicial maximal subcircuit of $\H$ is denoted by $\SMS(\H)$. With this terminology, the notion of chordal clutter is defined in \cite{SBCC} as follows.
	
\begin{definition}
	A $d$-uniform clutter $\H$ is a {\it chordal clutter} if either $E(\H)=\emptyset$ or there exists a $e\in\SMS(\H)$ such that $\H - e$ is again a chordal clutter.

    A $d$-uniform clutter $\H$ is said to be co-chordal if $\H^c$ is a chordal clutter.
\end{definition}

\subsection{Higher independence and the associated clutter}
	
Let $G$ be a graph on the vertex set $V(G)$ with circuit set $E(G)$. A subset $W$ of $V(G)$ is said to be {\it $r$-independent} if each connected component of $G[W]$ has vertex cardinality at most $r-1$. The collection of all $r$-independent sets forms the {\it $r$-independence complex} of $G$, and we denote it by $\Indr$. Note that $\mathrm{Ind}_2(G)$ is the classical {\it independence complex} $\mathrm{Ind}(G)$ of $G$. Moreover, for each $r\ge 2$, one can see that $I_{\Indr}$ is the square-free monomial ideal generated by the monomials associated to the connected induced subgraphs of $G$ with vertex cardinality $r$. In other words, $I_{\Indr}$ is the circuit ideal of the clutter $\Con_r(G)$, where $V(\Con_r(G))=V(G)$ and $E(\Con_r(G))=\{W\subseteq V(G): |W|=r\text{ and } G[W]\text{ is connected}\}$. The clutter $\Con_r(G)$ is also called the {\it $r$-connected clutter} of $G$. Note that the independence complex of $\Con_r(G)$ is the $r$-independence complex $\Indr$. 
	
For the $r$-independence complexes $\Indr$ it follows from the definition that $(\Indr)^{\vee}=\{F^c\mid F\in E(\Con_r(G))\}$. In \Cref{sec-3}, we define the simplicial complex $\SrAG$ for any subset $A$ of V(G), and it follows from the construction that $\Sigma_{r}(\emptyset, G)=(\Indr)^{\vee}$, which we also denote by $\Sigma_{r}(G)$. For the $r$-connected clutter $\Con_r(G)$, the induced matching number is denoted by $\gamma_r(G)$. Thus, $\gamma(\Con_r(G))=\gamma_r(G)$.  We say that a graph $G$ is {\it $r$-gap-free} if $\gamma_r(G)\le 1$. Observe that each gap-free graph is also $r$-gap-free. In other words, $\gamma_1(G)\le 1$ implies $\gamma_r(G)\le 1$. As we shall see in \Cref{Section 4 Froberg}, the $r$-gap-free condition is sometimes equivalent to the vertex decomposability property of the simplicial complex $\Sigma_{r}(G)$, for various well-known classes of graphs $G$.
	
\subsection{Stanley-Reisner ideal} Let $\Delta$ be a simplicial complex on the vertex set $V(\Delta)=\{x_1,\ldots,x_n\}$ and let $R$ denote the polynomial ring $\mathbb K[x_1,\ldots,x_n]$ over a field $\mathbb K$. 
It is well-known that the Stanley-Reisner ideal $I_{\Delta}$ of $\Delta$ is the square-free monomial ideal $\l\x_{F}:=\prod_{x_i\in F}x_i\mid F\notin\Delta\r$ in the polynomial ring $R$. 
Restricting to the complex $\Indr$ we obtain that 
\[
	I_{\Indr}=\l\x_{F}\mid F\subseteq V(G), |F|=r\text{ and }G[F] \text{ is connected }\r.
\]
Observe that the minimal generators of $I_{\Indr}$ correspond to the $r$-connected subsets of $G$, and more specifically, it is the circuit ideal of the $r$-connected clutter of $G$ defined above. 
Because of these reasons, the ideal $I_{\Indr}$ is also called the $r$-connected ideal in \cite{LQCG} (see also \cite{2025connectedideals}).
	
Numerous combinatorial properties of $\Delta$ possess algebraic counterparts in $I_{\Delta}$. 
For example, the Eagon-Reiner theorem \cite[Theorem 3]{EagonReiner1998} asserts that a simplicial complex $\Delta$ is Cohen-Macaulay if and only if $I_{\Delta^{\vee}}$ possesses a linear resolution. 
Similarly, according to a theorem by Herzog et al. \cite[Theorem 1.4]{HerzogHibiZheng2004}, $\Delta$ is shellable if and only if $I_{\Delta^{\vee}}$ exhibits linear quotients. 
More recently, Moradi and Khosh-Ahang provided an algebraic formulation of the vertex decomposability property. 
Specifically, they demonstrated in \cite[Theorem 2.3]{MoradiAhang2016} that $\Delta$ is vertex decomposable if and only if $I_{\Delta^{\vee}}$ is a vertex splittable ideal. 
In the present article, we employ these correspondences to infer certain properties of $I_{\Indr}$.
Nevertheless, as our primary focus lies on the properties of $\Indr$ and its Alexander dual, we do not elaborate on these algebraic properties and instead direct interested readers to the aforementioned works and the references therein for detailed definitions and further discussions on these topics.

    
\section{Co-connected complexes and their topology}\label{sec-3}
	
In this section, we formally define the simplicial complexes central to our study, together with their associated clutters, and examine their relevant topological and combinatorial properties. Throughout, we use the symbol $G$ to denote a graph. We begin by introducing the notion $r$-support of a subset $A$ of $V(G)$, for a positive integer $r$:
	\[
	\mathrm{Supp}_r(A, G) = \{ F \subseteq V(G)\setminus A : |F| = r,\  \text{$G[F \cup A]$ is connected} \}.
	\]
\begin{definition}
Given a positive integer $r$ and a subset $A$, the $r$-co-connected complex with respect to $A$, denoted by $\Sigma_r(A,G)$ is defined as
\[
	\Sigma_r(A,G) = \langle F^{c} \setminus A : F \in \mathrm{Supp}_r(A,G) \rangle ,
\]
where $F^{c}$ denotes the complement of $F$ in $V(G)$.    
\end{definition}

We additionally define $\Sigma_r(G)$ as the simplicial complex $\Sigma_r(\emptyset, G)$. 
It follows from the definition that for any positive integer $r$, the condition $\Supp_r(\emptyset,G) = E(\Con_r(G))$ holds, and for each $A$, we observe $V(\Sigma_r(A,G)) \subseteq V(G \setminus A)$. 
Moreover, $\Sigma_r(\emptyset, G)=(\mathrm{Ind}_r(G))^{\vee}$.

Next, we define the clutter $\D_r(A,G)$, which is related to the complex $\Sigma_r(A,G)$ through Alexander duality. Let
\begin{align*}
V(\D_r(A,G))&=V(G)\setminus A,\\
   E(\D_r(A,G)) &= \{\,F \subseteq V(G)\setminus A : |F|=r \ \text{and}\ G[A \cup F]\ \text{is disconnected}\}.
\end{align*}
It is easy to see that $\D_r(\emptyset,G)=(\Con_r(G))^c$ and we define $\D_r(\emptyset,G) = \D_r(G)$. Furthermore, for every $A \subseteq V(G)$ we have $(\mathrm{Ind}(\D_r(A,G)^c))^{\vee}=\Sigma_r(A,G)$. 

\begin{example}\label{example 1}
Let $G$ be the graph in \Cref{fig:chordal graph}. Take $r=2$ and $A=\{x_1\}$.

\begin{figure}[H]
		\centering
		\begin{tikzpicture}
			[scale=0.35, vertices/.style={draw, fill=black, circle, inner sep=1.5pt}]
			\node[vertices, label=above:{$x_1$}] ($x_1$) at (0,2)  {};
			\node[vertices, label=below:{$x_2$}] ($x_2$) at (0,-2)  {};
			\node[vertices, label=above:{$x_3$}] ($x_3$) at (-2,0)  {};
			\node[vertices, label=left:{$x_4$}] ($x_4$) at (-4,0)  {};
			\node[vertices, label=above:{$x_5$}] ($x_5$) at (2,0)  {};
			\node[vertices, label=right:{$x_6$}] ($x_6$) at (4,0)  {};
			
			\foreach \to/\from in {$x_1$/$x_2$,$x_1$/$x_3$,$x_1$/$x_5$,$x_2$/$x_3$,$x_2$/$x_5$,$x_3$/$x_4$,$x_5$/$x_6$}
			\draw [-] (\to)--(\from);
		\end{tikzpicture}\caption{}\label{fig:chordal graph}
	\end{figure}    
We first compute the $2$-support of $A$:
\begin{align*}
\mathrm{Supp}_2(A,G) & = \{ F \subseteq V(G)\setminus A : |F|=2,\; G[F \cup A]\ \text{is connected} \} \\ 
& = \{\{x_2,x_3\},\ \{x_2,x_5\},\ \{x_3,x_4\},\ \{x_3,x_5\},\ \{x_5,x_6\}\} .
\end{align*}

\noindent
Hence,
\begin{align*}
\Sigma_2(A,G)
& = \langle F^c \setminus A : F \in \mathrm{Supp}_2(A,G) \rangle\\
& = \langle \{x_4,x_5,x_6\},\ \{x_3,x_4,x_6\},\ \{x_2,x_5,x_6\},\ \{x_2,x_4,x_6\},\ \{x_2,x_3,x_4\} \rangle .
\end{align*}
Thus $\Sigma_2(A,G)$ is a simplicial complex on $V(G)\setminus A = \{x_2,x_3,x_4,x_5,x_6\}$ with the above facets. Moreover, 
\[
E(\D_2(A,G))=\{\{x_2,x_4\},\{x_4,x_5\}, \{x_4,x_6\},\{x_3,x_6\},\{x_2,x_6\}\}.
\]
\hfill$\diamond$
\end{example}

\medskip
\medskip
Our first main objective in this section is to prove that $\D_r(A,G)$ is a chordal clutter whenever $A$ is a nonempty subset of $V(G)$ with the property that $G[A]$ is connected. To this end, we introduce, for each $r$-uniform subclutter $\H$ of $\D_r(A,G)$ and each vertex $v$ of $\H$, a new clutter $\H_v$ defined as follows:
\begin{align*}
  V(\H_v) &= V(\H) \setminus \{v\},\\
  E(\H_v) &= \{\E \setminus \{v\} :\E \in E(\H) \; \text{and} \; v \in \E\}.
\end{align*}

Note that $\H_v$ is an $(r-1)$-uniform clutter. Now, for the clutter $\H_v$ we can deduce the following.
\begin{lemma}\label{nbhd equality}
    If $e \in \MS(\H_v)$, then $\N_{\H_v}[e] \;=\; \N_{\H}[e \cup \{v\}] \setminus \{v\}$.
\end{lemma}
\begin{proof}
    Let $e \in \MS(\H_v)$.  
    For any $x \in V(\H_v) \setminus e$, we have
    \[
        e \cup \{x\} \in E(\H_v) 
        \quad\Longleftrightarrow\quad 
        e \cup \{x,v\} \in E(\H).
    \]  
    Thus, $\N_{\H_v}[e] = \N_{\H}[e \cup \{v\}] \setminus \{v\}$.
\end{proof}

Observe that for a vertex $v \in V(\D_r(A,G))$, we can always decompose the circuit set of the $\D_r(A,G)$ as follows:
\begin{equation}\label{eq:2}
    E(\D_r(A,G)) = \{F \cup \{v\} : F \in E(\D_{r-1}(A \cup \{v\},G))\} \bigsqcup E(\D_r(A, G\setminus \{v\})).
\end{equation}
Based on this observation, one can verify that to establish the chordality of $\D_r(A,G)$, it suffices to construct a sequence of simplicial maximal subcircuits in the subclutters of $\D_{r-1}(A \cup \{v\}, G)$, as demonstrated in \Cref{chordal 1}. Before proceeding, we record the following short lemma.
	\begin{lemma}\label{connectivity lemma}
		Let $G$ be a connected graph with $|V(G)| \geq 2$. Then, the set $\{v \in V(G): G \setminus v \text{ is connected}\}$ has cardinality at least two.
	\end{lemma}
	\begin{proof} 
		Since $G$ is a connected graph, $G$ has one spanning tree, say $T$. Since $|V(G)|\ge 2$, we see that $T$ has at least two leaf vertices. Now, observe that $G\setminus x$ is a connected graph for each leaf $x$ of $T$. Thus, we have our result.
	\end{proof}

\begin{lemma}\label{chordal 1}
Let $r \ge 2$ be a positive integer and let $A \subseteq V(G)$ be a non-empty set such that $G[A]$ is connected.
Suppose $v \in N_G(A)$ so that $\D_{r-1}(A \cup \{v\}, G)$ is chordal. 
Then there exists a sequence $e_1, \ldots, e_t$ such that for each $i\in[t]$,
\[
e_i \in \SMS(\D_{i-1}), \text{~where~}
\;
\D_0=\D_r(A,G),\ \D_i=\D_{i-1}-e_i, \text{~and~}
\D_t=\D_r(A,G\setminus\{v\}).
\]

\end{lemma}
\begin{proof}
If $\D_{r-1}(A \cup \{v\}, G)$ is the empty clutter, the result follows immediately from the \Cref{eq:2}. Assume that $\D_{r-1}(A \cup \{v\}, G)$ is a non-empty chordal clutter. Then there exists a sequence $e'_1, \ldots, e'_t$
with
\[e'_i \in \SMS(\D'_{i-1}), \qquad \D'_0 = \D_{r-1}(A \cup \{v\},G), \qquad\D'_i = \D'_{i-1} - e_i',\] 
such that $\D'_t$ is the empty clutter. 

\medskip
\noindent
\textbf{Claim}
    Then $e_1 , \ldots, e_t $ is a required sequence, where $e_i = e'_i \cup \{v\}$ for all $i \in [t]$.
    
\medskip

\noindent
\textbf{Proof of Claim:}
Note that, for each $i \in \{0\} \cup [t]$ and the given $v$, the clutter $\D'_i$ is the $\mathcal{H}_v$-type subclutter of $\mathcal{D}_i$. 
For $i\in [t]$, by \Cref{nbhd equality},
\begin{equation}\label{eq:3}
\N_{\D'_{i-1}}[e'_{i}] = \N_{\D_{i-1}}[e'_{i} \cup \{v\}] \setminus \{v\} = \N_{\D_{i-1}}[e_{i}] \setminus \{v\}.
\end{equation}
Let $B$ be an $r$-subset of $\N_{\D_{i-1}}[e'_{i} \cup \{v\}]$.  
We want to show $B \in E(\D_{i-1})$.

If $v \in B$, then 
\[
B\setminus\{v\} \subseteq \N_{\D'_{i-1}}[e'_{i}],
\]
and since $e'_{i}\in \SMS(\D'_{i-1})$, we have $B\setminus\{v\}\in E(\D'_{i-1})$, hence $B\in E(\D_{i-1})$.
Assume $v \notin B$.
If $B\notin E(\D_{i-1})$ that is $B \in E({(\D_{i-1}})^c)$, then $G[A \cup B]$ is connected, because \[E((\D_{i-1})^c) = E({(\D_r(A,G)})^c) \cup \big(\bigcup\limits_{j=1}^{i-1} W_j \big),\] where $W_j = \{F \supseteq e_j : |F| = r, G[F \cup A] \text{~is disconnected}\}$. 

Now, we aim to show that there exists $x \in B$ such that $G[(A \cup B) \setminus \{x\}]$ is connected. 
Let $C$ be a connected component of $G[B]$.  
If $|C|=1$ and $C=\{y\}$, then
\[
G[(A\cup B)\setminus\{y\}]
\]
remains connected, since both $G[A]$ and $G[A\cup B]$ are connected.

If $|C|\ge 2$, then by the \Cref{connectivity lemma}, $C$ contains two vertices $y_1,y_2$ such that $G[C\setminus\{y_j\}]$ is connected for $j=1,2$.  
If $y_1\in N_G(A)$, then 
\[
G[(A\cup B)\setminus\{y_2\}]
\]
is connected; otherwise the same holds with $y_1$.  
Thus, there exists $x \in B$ such that $G[(A\cup B)\setminus\{x\}]$
is connected.
Consequently, $G[A\cup\{v\}\cup(B\setminus\{x\})]$ is connected.  
But by \Cref{eq:3}, $B\setminus\{x\}\subseteq \N_{\D'_{i-1}}[e'_{i}]$ since $v \notin B$. Now $e'_{i}\in \SMS(\D'_{i-1})$ imply $B\setminus\{x\}\in E(\D'_{i-1})$,
contradicting the connectivity above.  
Thus $B\in E(\D_{i-1})$. Finally, since $\D'_t$ is empty, we obtain from the \Cref{eq:2} that $\D_t=\D_r(A,G\setminus\{v\})$. 
 This completes the proof.
\end{proof}

Intuitively, the idea of the proof of the following theorem (and subsequent similar results) is to inductively eliminate vertices whose neighborhoods preserve connectivity, mimicking perfect elimination in chordal graphs.

\begin{theorem}\label{chordal 2}
    Let $r \ge 1$ be a positive integer and let $A \subseteq V(G)$ be a non-empty set such that $G[A]$ is connected. Then $\D_r(A,G)$ is a chordal clutter. 
\end{theorem}
\begin{proof}
Let $|V(G)|=n$. We use induction on $r+n$. By the given hypothesis $r\ge 1$, and since $A\neq\emptyset$, we must have $n\ge 1$.
		
		\medskip
		\noindent
		\textbf{Base case: }If $r+n=2$ then $|V(G)|=1$ and $A = V(G)$. Thus $\D_1(A,G)$ is the empty set. Therefore, it is trivially a chordal clutter.
		
		\medskip
		\noindent
		\textbf{Induction hypothesis: } Let $r,n$ be positive integers such that $r+n\ge 3$. We assume that if $r'$ and $n'$ are two positive integers such that $r'+n'<r+n$ and there exists a graph $G'$ with $|V(G')|=n'$, then $\D_{r'}(A',G')$ is a chordal clutter for any non-empty set $A'\subseteq V(G')$ such that $G'[A']$ is connected.
		
		\medskip
		\noindent
		\textbf{Inductive Step :} We divide this step into the following two cases:
		
		\medskip
		\noindent
		\textbf{Case I :} $r=1$. In this case, we have $E(\D_1(A,G)) = V(G) \setminus N_G[A]$ is the $1$-uniform complete graph on $V(G) \setminus N_G[A]$, which is chordal. 
		
		\medskip
		\noindent
		\textbf{Case II :} $r\ge 2$. Observe that the complete $r$-uniform clutter on $V(G) \setminus N_G[A]$ appears as an induced subclutter of $\D_r(A,G)$. In particular, if $\D_r(A,G)$ coincides with the complete $r$-uniform clutter on $V(G) \setminus N_G[A]$, then $\D_r(A,G)$ is  chordal. 
Now suppose there exists $F \in E(\D_r(A,G))$ such that $F \cap N_G(A) \neq \emptyset$, and let $v \in F \cap N_G[A]$. Now by \Cref{eq:2}, we can decompose
\[
E(\D_r(A,G)) = \{ F \cup \{v\} : F \in E(\D_{r-1}(A \cup \{v\}, G))\} \bigsqcup E(\D_r(A, G \setminus \{v\})).
\]
Since $(r-1) + |V(G)| < r + |V(G)|$, the induction hypothesis implies that $\D_{r-1}(A \cup \{v\},G)$ is chordal. 
By \Cref{chordal 1} there exists a sequence $e_1, \ldots, e_t$
with each $e_i \in \SMS(\D_{i-1})$, where $\D_0 = \D_r(A,G)$ and $\D_i = \D_{i-1} - e_i$, such that 
\[
\D_t = \D_r(A,\, G \setminus \{v\}).
\] Again observe that $r + |V(G \setminus \{v\})| < r + |V(G)|$, by the induction hypothesis, $\D_r(A, G\setminus \{v\})$ is chordal. Hence, the result follows.  
\end{proof}

\medskip
\begin{example}
We illustrate \Cref{chordal 4} using the \Cref{example 1}. Let $G$ be the graph in \Cref{fig:chordal graph}. Take $r=2$ and $A=\{x_1\}$. Recall that
\[
E(\D_2(A,G))
=
\{\{x_2,x_4\},\{x_4,x_5\},\{x_4,x_6\},\{x_3,x_6\},\{x_2,x_6\}\}.
\]

Consider $x_3 \in N_G(x_1)$. Then
\begin{align*}
E(\D_2(\{x_1\},G))
&=
\{F \cup \{x_3\} : F \in E(\D_1(\{x_1,x_3\},G))\}
\;\cup\;
E(\D_2(\{x_1\},G\setminus\{x_3\})) \\
&=
\{\{x_3,x_6\}\}
\;\cup\;
\{\{x_2,x_4\},\{x_4,x_5\},\{x_4,x_6\},\{x_2,x_6\}\}.
\end{align*}

\noindent
Since $\D_1(\{x_1,x_3\},G)$ is a chordal clutter, it remains to show that
$\D_2(\{x_1\},G\setminus\{x_3\})$ is chordal.

Now, $x_2 \in N_{G\setminus\{x_3\}}(x_1)$. Hence,
\begin{align*}
E(\D_2(\{x_1\},G\setminus\{x_3\}))
&=
\{F \cup \{x_2\} : F \in E(\D_1(\{x_1,x_2\},G\setminus\{x_3\}))\} \\
&\qquad \cup\;
E(\D_2(\{x_1\},G\setminus\{x_3,x_2\})) \\
&=
\{\{x_2,x_4\},\{x_2,x_6\}\}
\;\cup\;
\{\{x_4,x_5\},\{x_4,x_6\}\}.
\end{align*}

\noindent
Since $\D_1(\{x_1,x_2\},G\setminus\{x_3\})$ is chordal, it remains to verify that
$\D_2(\{x_1\},G\setminus\{x_3,x_2\})$ is chordal.

Observe that $x_5 \in N_{G\setminus\{x_3,x_2\}}(x_1)$. Thus,
\begin{align*}
E(\D_2(\{x_1\},G\setminus\{x_3,x_2\}))
&=
\{F \cup \{x_5\} : F \in E(\D_1(\{x_1,x_5\},G\setminus\{x_3,x_2\}))\} \\
&\qquad \cup\;
E(\D_2(\{x_1\},G\setminus\{x_3,x_2,x_5\})) \\
&=
\{\{x_4,x_5\}\}
\;\cup\;
\{\{x_4,x_6\}\}.
\end{align*}

\noindent
Both $\D_1(\{x_1,x_5\},G\setminus\{x_3,x_2\})$ and
$\D_2(\{x_1\},G\setminus\{x_3,x_2,x_5\})$ are chordal.

Therefore, $\D_2(\{x_1\},G)$ is a chordal clutter.
\hfill$\diamond$
\end{example}

\medskip
Thus far, we have shown that the $ r$-uniform clutter, $\D_r(A,G)$ is always a chordal clutter, provided that $A$ is a non-empty subset of $V(G)$ and $G[A]$ is connected. Note that if $A = \emptyset$, then it is not necessarily true that $\D_r(G)$ is chordal (see \Cref{Section 4 Froberg}). In the next part, we provide a sufficient condition under which $\D_r(G)$ is chordal.

\begin{lemma}\label{chordal 3}
Let $r \ge 2$ be a positive integer. Suppose $v \in V(G)$ satisfies $\Con_r\bigl(G \setminus N_{G}[v]\bigr) = \emptyset$ and $\D_{r-1}(\{v\}, G)$ is chordal. Then there exists a sequence $e_1, \ldots, e_t$
such that for each $i\in[t]$,
\[e_i \in \SMS(\D_{i-1}), \text{~where,~} \; \D_0 = \D_r(G), \qquad\D_i = \D_{i-1} - e_i, \text{~and~}
\D_t = \D_r(\, G \setminus \{v\}).
\]
\end{lemma}
\begin{proof}

If $\D_{r-1}(\{v\}, G)$ is the empty clutter, the result follows immediately from the \Cref{eq:2}. Assume that $\D_{r-1}(\{v\}, G)$ is a non-empty chordal clutter. There exists a sequence $e'_1, \ldots, e'_t$
with 
\[e'_i \in \SMS(\D'_{i-1}), \qquad \D'_0 = \D_{r-1}(\{v\},G), \qquad \D'_i = \D'_{i-1} - e_i',\] such that $\D'_t$ is an empty clutter. 

\medskip
\noindent
\textbf{Claim}
Then $e_1 , \ldots, e_t $ is a required sequence, where $e_i = e'_i \cup \{v\}$ for all $i \in [t]$.  

\medskip
\noindent
\textbf{Proof of Claim:}
Note that, for each $i \in \{0\} \cup [t]$ and the given $v$, the clutter $\D'_i$ is the $\mathcal{H}_v$-type subclutter of $\mathcal{D}_i$. For $i\in[t]$, by \Cref{nbhd equality}, 
\begin{equation}\label{eq:4}
\N_{\D'_{i-1}}[e_i'] = \N_{\D_{i-1}}[e_i' \cup \{v\}] \setminus \{v\} = \N_{\D_{i-1}}[e_i] \setminus \{v\}.
\end{equation}
Let $B$ be an $r$-subset of $\N_{\D_{i-1}}[e'_i \cup \{v\}]$.  
We want to show $B \in E(\D_{i-1})$.

If $v \in B$, then 
\[
B\setminus\{v\} \subseteq \N_{\D'_{i-1}}[e'_i],
\]
and since $e'_i \in \SMS(\D'_{i-1})$, we have $B\setminus\{v\}\in E(\D'_{i-1})$, hence $B\in E(\D_{i-1})$. 
Assume $v \notin B$. If $B \notin E(\D_{i-1})$ that is $B \in E((\D_{i-1})^c)$ then $G[B]$ is connected, because \[E((\D_{i-1})^c) = E({\D_r(G)}^c) \cup \big(\bigcup\limits_{j=1}^{i-1} W_j \big),\] where $W_j = \{F \supseteq e_j : |F|=r, G[F] \text{~ is disconnected ~} \}$.
Now, we aim to show that there exists $x \in B$ ssuch that $G[B \setminus \{x\}]$ is connected. 
By the choice of $v$, we have $B \cap N_G[v] \neq \emptyset$. 
Also, by \Cref{connectivity lemma}, we see that $G[B]$ has at least two vertices, say $y_1,y_2$ such that $G[B \setminus\{y_i\}]$ is connected for each $i\in\{1,2\}$. If $y_1\in N_G(v)$, then one can see that $G[B \setminus \{y_2\}]$ is connected. On the other hand, if $y_1 \notin N_G(v)$, then  $G[(B \setminus \{y_1\}]$ is connected, by similar reason. Thus, there exists a vertex $x \in B$ such that $G[B \setminus \{x\}]$ is connected. Consequently, $G[\{v\} \cup B \setminus \{x\}]$ is connected. 
But by \Cref{eq:4}, $B \setminus \{x\} \subseteq \N_{\D'_{i-1}}[e_i]$ since $v \notin B$. Now,  $e_i' \in \SMS(\D'_{i-1})$ imply $B \setminus \{x\} \in E(\D'_{i-1})$, which contradicts the connectivity above. Therefore, $B \in E(\D_{i-1})$. After $t$ steps, $\D'_t$ is an empty clutter, we obtain from the \Cref{eq:2} that $\D_t = \D_r(G\setminus \{v\})$.
This completes the proof.
\end{proof}

\begin{theorem}\label{chordal 4}
    Let $G$ be a graph and let $r \ge 2$ be a positive integer. 
    Consider the following two statements:
    \begin{enumerate}
        \item There exists an ordering of the vertices $x_1, x_2, \ldots, x_m$ of $G$ such that:
        \begin{enumerate}
            \item For each $i \in [m-1] \cup \{0\}$,
            \[
                E\left(\Con_r\bigl(G_i \setminus N_{G_i}[x_{i+1}]\bigr)\right) = \emptyset,
            \]
            where $G_0 = G$ and, $G_i = G \setminus \{x_1, \ldots, x_i\}$ for each $i \ge 1$.
            
            \item $m$ is the least positive integer such that $\D_r(G \setminus \{x_1,x_2,\ldots,x_m\}))$ is the empty clutter.
        \end{enumerate}
        \item The $r$-uniform clutter $\D_r(G)$ is chordal. 
    \end{enumerate}
    Then the condition \textup{(1)} implies \textup{(2)}.
\end{theorem}
\begin{proof} 
    We proceed by induction on $|V(G)|$. If $|V(G)| \le r$ then $\D_r(G)$ is either the empty clutter or has a single circuit, which is trivially a chordal clutter. Assume that, if $|V(G')| < |V(G)|$ and $G'$ satisfies the condition \emph{(1)} then the $r$-uniform clutter $\D_r(G')$ is chordal. Now, consider the vertex $x_1 \in V(G)$ from the sequence of \emph{(1)} for which $E(\Con_r(G \setminus N_G[x_1])) = \emptyset$. By \Cref{eq:2} we have
    \[E(D_r(G)) = \{F \cup \{x_1\} : F \in E(\D_{r-1}(\{x_1\},G))\} \bigsqcup E(\D_r(G\setminus \{x_1\})).\] By \Cref{chordal 2}, $\D_{r-1}(\{x_1\},G)$ is chordal clutter and thus, in view of \Cref{chordal 3}, there exists a sequence $e_1, \ldots, e_t$
with each 
\[e_i \in \SMS(\D_{i-1}), \qquad \D_0 = \D_r(G), \qquad \D_i = \D_{i-1} - e_i,\] such that 
\[
\D_t = \D_r(\, G \setminus \{x_1\}).
\] As, $|V(G \setminus \{x_1\}| < |V(G)|$, and $G \setminus \{x_1\}$ satisfies the condition \emph{(1)}, by the induction hypothesis we conclude that $\D_r(G \setminus \{x_1\})$ is a chordal clutter. Hence, the $r$-uniform clutter $\D_r(G)$ is chordal. 
\end{proof}

\medskip
\begin{example}
We illustrate \Cref{chordal 4} using the graph in \Cref{fig:chordal graph}.
Let $r=3$.

We first compute the circuit set of the clutter $\Con_3(G)$:
\[
E(\Con_3(G))
=
\{ F \subseteq V(G) : |F|=3 \text{ and } G[F] \text{ is connected} \}.
\]

The $3$-subsets $F$ for which $G[F]$ is connected are
\[
\begin{aligned}
\{
&\{x_1,x_2,x_3\},\{x_1,x_2,x_5\},\{x_1,x_3,x_4\},\{x_1,x_3,x_5\},\{x_1,x_5,x_6\},\\
&\{x_2,x_3,x_4\},\{x_2,x_3,x_5\},\{x_2,x_5,x_6\}
\}.
\end{aligned}
\]

Thus these sets form the circuits of the clutter $\Con_3(G)$.

The circuit set of $\D_3(G)$ is
\begin{align*}
\{
&\{x_1,x_2,x_4\},\{x_1,x_2,x_6\},\{x_1,x_3,x_6\},\{x_1,x_4,x_5\},\{x_1,x_4,x_6\},\\
&\{x_2,x_3,x_6\},\{x_2,x_4,x_5\},\{x_2,x_4,x_6\},\{x_3,x_4,x_5\},\{x_3,x_4,x_6\},\\
&\{x_3,x_5,x_6\},\{x_4,x_5,x_6\}
\}.
\end{align*}

We now construct an ordering of the vertices satisfying
Theorem~\ref{chordal 4}(1).

\medskip
\noindent
\textbf{Step 1.}
For the vertex $x_1$, every circuit of $\Con_3(G)$ contains a vertex from $N_G[x_1]$.
Hence,
\[
E(\Con_3(G \setminus N_G[x_1])) = \emptyset.
\]

\medskip
\noindent
\textbf{Step 2.}
Consider the graph $G_1 = G \setminus \{x_1\}$. Then
\[
E(\Con_3(G_1))
=
\{\{x_2,x_3,x_4\},\{x_2,x_3,x_5\},\{x_2,x_5,x_6\}\}.
\]
Each of these sets intersects $N_{G_1}[x_2]$, and therefore
\[
E(\Con_3(G_1 \setminus N_{G_1}[x_2]))=\emptyset.
\]

\medskip
\noindent
\textbf{Step 3.}
Let $G_2 = G \setminus \{x_1,x_2\}$. Then
\[
E(\D_3(G_2))
=
\{\{x_3,x_4,x_5\},\{x_3,x_4,x_6\},\{x_3,x_5,x_6\},\{x_4,x_5,x_6\}\}.
\]
Moreover,
\[
E(\Con_3(G_2))=\emptyset.
\]

For any vertex $v \in \{x_3,x_4,x_5,x_6\}$,
\[
E(\Con_3(G_2 \setminus N_{G_2}[v]))=\emptyset.
\]
Choose $v=x_4$. Then
\[
E(\D_3(G\setminus\{x_1,x_2,x_4\}))=\{\{x_3,x_5,x_6\}\}.
\]

\medskip
\noindent
\textbf{Step 4.}
Let $G_3 = G\setminus\{x_1,x_2,x_4\}$. For $x_6$,
\[
E(\Con_3(G_3 \setminus N_{G_3}[x_6]))=\emptyset,
\]
and
\[
E(\D_3(G\setminus\{x_1,x_2,x_4,x_6\}))=\emptyset.
\]

Thus the ordering $(x_1,x_2,x_4,x_6)$ satisfies condition~\emph{(1)} of
\Cref{chordal 4}. Consequently, $\D_3(G)$ is a chordal clutter.
\hfill$\diamond$
\end{example}

\medskip
Our objective now is to determine the conditions under which the complex $\Sigma_r(A,G)$ is vertex decomposable. 
We begin by identifying a shedding vertex. 
Recall that for a simplicial complex $\Delta$, an element $x\in V(\Delta)$ is called a {\it shedding vertex} of $\Delta$ if no face of $\link_{\Delta}(x)$ constitutes a facet of $\del_{\Delta}(x)$. 
Alternatively, $x$ is considered a shedding vertex of $\Delta$ if and only if every facet of $\del_{\Delta}(x)$ is a facet of $\Delta$. 
This characterization is used in the next proposition to establish a combinatorial condition for identifying a shedding vertex of $\Sigma_r(A,G)$.
	
\begin{proposition}\label{Eq. Condition for a shedding vertex}
Let $G$ be a graph, $r$ a fixed positive integer, and $x \in V(\Sigma_r(A,G))$. The following conditions are equivalent:
\begin{enumerate}
	\item[(i)] $x$ is a shedding vertex of $\Sigma_r(A,G)$.		
	\item[(ii)] For every $F \in \Supp_r(A,G)$ with $x \notin F$, there exists a vertex $y \in F$ such that $(F \setminus \{y\}) \cup \{x\} \in \mathrm{Supp}_r(A,G)$.
\end{enumerate}
\end{proposition}

\begin{proof}	
The goal is to show \emph{(i) $\Rightarrow$ (ii)}. 
Let $x$ be a shedding vertex of $\Sigma_r(A,G)$, and $F \in \mathrm{Supp}_r(A,G)$ such that $x \notin F$. 
Note $x \in F^{c}\setminus A$, where $F^c\setminus A$ is a facet of $\Sigma_r(A,G)$. 
Also, $F^{c} \setminus (A \cup \{x\})$ is a face of $\del_{\Sigma_r(A,G)}(x)$, and since $x$ is a shedding vertex of $\Sigma_r(A,G)$, $F^{c} \setminus (A \cup \{x\})$ is not a facet of $\del_{\Sigma_r(A,G)}(x)$. 
In particular, there exists some $H \in \mathrm{Supp}_r(A,G)$ with $x \in H$ such that 
\[
F^{c} \setminus (A \cup \{x\}) \subseteq H^{c} \setminus A.
\]
Since $H\cap A=\emptyset$ we must have $H \subset F \cup \{x\}$. Furthermore, since $|F|=|H|=r$ and $x\in H\setminus F$, we find that $F \setminus H = \{y\}$ for some $y \in V(G)$. Thus $H=(F \setminus \{y\}) \cup \{x\} \in \mathrm{Supp}_r(A,G)$, as desired.
		
To prove \emph{(ii) $\Rightarrow$ (i)}, it suffices to verify that every facet of $\mathrm{del}_{\Sigma_r(A,G)}(x)$ is also a facet of $\Sigma_r(A,G)$. In other words, we need to show that $\mathcal{F}(\mathrm{del}_{\Sigma_r(A,G)}(x)) \subseteq \mathcal{F}(\Sigma_r(A,G))$.
Observe that
		\[
		\mathcal{F}(\mathrm{del}_{\Sigma_r(A,G)}(x)) \subseteq \{F^{c} \setminus (A \cup \{x\}) : F \in \mathrm{Supp}_r(A,G)\} 
		\]
Thus, we want to show that  $\mathcal{F}(\mathrm{del}_{\Sigma_r(A,G)}(x)) = \{F^{c}\setminus A : F \in \Supp_r(A,G) \; \text{and} \; x\in F\}$.
		
If $x\in F$, then $F^{c} \setminus (A \cup \{x\}) = F^{c} \setminus A$, leading to $F^{c} \setminus A\in \mathcal{F}(\Sigma_r(A,G))$. 
Now, consider when $x \notin F$. 
Based on the assumption, there is a $y\in V(G)$ such that $H := (F \setminus \{y\}) \cup \{x\} \in \mathrm{Supp}_r(A,G)$. 
Given $x\notin F\cup A$, it follows that $y\in F$. Note $H^{c} \setminus A = (F^{c} \cup \{y\}) \setminus (A \cup \{x\})$, thus $F^{c} \setminus (A \cup \{x\}) \subseteq H^{c} \setminus A$. 
Since $H^c\setminus A\in \F( \del_{\Sigma_r(A,G)}(x))$, we have $F^c\setminus (A\cup\{x\})\notin \mathcal{F}(\mathrm{del}_{\Sigma_r(A,G)}(x))$, resulting in $\mathcal{F}(\mathrm{del}_{\Sigma_r(A,G)}(x)) = \{F^{c}\setminus A : F \in \Supp_r(A,G) \; \text{and} \; x\in F\}$, completing the proof.   
\end{proof}

\medskip
\begin{example}
We illustrate \Cref{Eq. Condition for a shedding vertex} using the \Cref{example 1}.  
Let $G$ be the graph in \Cref{fig:chordal graph}, take $r=2$, and
$A=\{x_1\}$. Recall that
\[
\Supp_2(A,G)
= \bigl\{
\{x_2,x_3\},\{x_2,x_5\},\{x_3,x_4\},\{x_3,x_5\},\{x_5,x_6\}
\bigr\}.
\]

\noindent
We verify condition~\emph{(ii)} of \Cref{Eq. Condition for a shedding vertex}
for the vertex $x_3 \in V(\Sigma_2(A,G))$.

\noindent
Consider $F \in \Supp_2(A,G)$ with $x_3 \notin F$.

\begin{itemize}
    \item If $F=\{x_2,x_5\}$, then choosing $y=x_5$ gives
    \[
    (F\setminus\{x_5\})\cup\{x_3\}
    =\{x_2,x_3\}\in \Supp_2(A,G).
    \]

    \item If $F=\{x_5,x_6\}$, then choosing $y=x_6$ yields
    \[
    (F\setminus\{x_6\})\cup\{x_3\}
    = \{x_3,x_5\} \in \Supp_2(A,G).
    \]
\end{itemize}

\noindent
Thus, for every $F\in \Supp_2(A,G)$ with $x_3\notin F$, there exists
$y\in F$ such that
\[
(F\setminus\{y\})\cup\{x_3\}\in \Supp_2(A,G).
\]
Hence, by \Cref{Eq. Condition for a shedding vertex},
$x_3$ is a shedding vertex of\/ $\Sigma_2(A,G)$.

\medskip

Next, we verify condition~\emph{(ii)} of \Cref{Eq. Condition for a shedding vertex}
for the vertex $x_4 \in V(\Sigma_2(A,G))$.

\noindent
Consider $F=\{x_5,x_6\}$. We have $x_4\notin F$, but
\[
(F\setminus\{x_5\})\cup\{x_4\}
= \{x_4,x_6\}\notin \Supp_2(A,G)
\]
and
\[
(F\setminus\{x_6\})\cup\{x_4\}
= \{x_4,x_5\}\notin \Supp_2(A,G).
\]

\noindent
Thus, there exists $F\in \Supp_2(A,G)$ with $x_4\notin F$ such that, for every
$y\in F$,
\[
(F\setminus\{y\})\cup\{x_4\}\notin \Supp_2(A,G).
\]
Hence, by \Cref{Eq. Condition for a shedding vertex},
$x_4$ is not a shedding vertex of\/ $\Sigma_2(A,G)$.
\hfill$\diamond$
\end{example}

\medskip
Next, we describe the link and the deletion of a vertex in simplicial complexes $\Sigma_r(A,G)$. 
If $\Delta$ is a simplicial complex involving $\F(\Delta)=\{\Gamma_1,\ldots,\Gamma_t\}$ and $x\in V(\Delta)$, the link of $x$ in $\Delta$ is defined by $\lk_{\Delta}(x)=\l \Gamma_i\setminus \{x\}\mid \Gamma_i\in \F(\Delta)\text{ and }x\in \Gamma_i\r$. 
Furthermore, if $x$ is a shedding vertex of $\Delta$, then the deletion of $x$ in $\Delta$ is $\del_{\Delta}(x)=\l \Gamma_i\mid \Gamma_i\in\F(\Delta) \text{ and } x\notin \Gamma_i\r$. 
Using these definitions, the following results can be shown.
    
	\begin{proposition}\label{Eq. conditions for lk and del}
		Let $G$ be a graph, $r$ a positive integer and $x \in V(\Sigma_r(A,G))$. Then the following statements hold:
		\begin{enumerate}
			\item For $r \geq 1$, $\lk_{\Sigma_r(A,G)}(x) = \Sigma_r(A,G \setminus x) =  \langle F^{c} \setminus A : F \in \mathrm{Supp}_{r}(A,G\setminus x) \rangle$.
			\item  If $x$ is a shedding vertex of $\Sigma_r(A,G)$, then
			\begin{enumerate}
				\item $\del_{\Sigma_r(A,G)}(x) = \Sigma_{r-1}(A \cup \{x\},G) =  \langle F^{c} \setminus (A\cup \{x\}) : F \in \mathrm{Supp}_{r-1}(A\cup \{x\},G) \rangle$ when $r\ge 2$.
				
				\item $G[A\cup\{x\}]$ is connected when $r=1$, and in this case, $\del_{\Sigma_1(A,G)}(x) =\Delta_{V(G)\setminus (A\cup\{x\})}$.
			\end{enumerate}
		\end{enumerate}
	\end{proposition}
	\begin{proof} 
		(1) By the above characterization of the link of a vertex, we obtain the following
		\[
		\lk_{\Sigma_r(A,G)}(x) = \langle F^c \setminus (A \cup \{x\}) : F \in \mathrm{Supp}_r(A, G),\ x \notin F \rangle.
		\]
		Observe that $x\in V(\SrAG)$ implies $A\subseteq V(G\setminus x)$. Thus we have a one-to-one correspondence between the two sets $\{F\subseteq V(G): x\notin F , |F|=r, F\cap A=\emptyset, G[F\cup A]\text{ is connected}\}$ and $\{F\subseteq V(G\setminus x): |F|=r, F\cap A=\emptyset, (G\setminus x)[F\cup A]\text{ is connected}\}$ via the inclusion map. Hence, $\Supp_r(A,G\setminus x)=\{F\in\Supp_r(A,G): x\notin F\}$. Consequently, 
		\begin{align*}
			\lk_{\Sigma_r(A,G)}(x) &=\langle F^c \setminus (A \cup \{x\}) : F \in \mathrm{Supp}_r(A, G),\ x \notin F \rangle\\
			&=\langle F^c\setminus A: F\in\Supp_r(A, G\setminus x) \rangle. 
		\end{align*} 
	
		(2) Since $x$ is a shedding vertex of $\Sigma_r(A,G)$, we have the following.
		\begin{align*}
			\del_{\Sigma_r(A,G)}(x) &= \langle F^c \setminus A : F \in \mathrm{Supp}_r(A, G), x \notin F^c\setminus A \rangle\\
			&=\l F^c \setminus A : F \in \mathrm{Supp}_r(A, G), x \in F\r.
		\end{align*}
		First, we consider the case $r \ge 2$. Let us consider the following two sets:
		\begin{align*}
			S_1&=\{F\subseteq V(G)\mid |F|=r,F\cap A=\emptyset,x\in F,G[F\cup A]\text{ is connected}\}\\
			S_2&=\{F\subseteq V(G)\mid |F|=r-1,F\cap (A\cup\{x\})=\emptyset,x\notin F,G[F\cup (A\cup\{x\})]\text{ is connected}\}.
		\end{align*}
		Observe that $S_1=\{F\in \Supp_r(A,G): x\in F\}$,  $S_2=\Supp_{r-1}(A\cup\{x\},G)$, and the map $F\mapsto F\setminus\{x\}$ is a bijection between $S_1$ and $S_2$. Moreover, $\l F^c\setminus A: F\in S_1\r=\l F^c\setminus (A\cup\{x\}): F\in S_2\r$. Thus,
		\[
		\del_{\Sigma_r(A,G)}(x) = \Sigma_{r-1}(A \cup \{x\},G).
		\]
		
		Now, suppose $r=1$. Then it follows from \Cref{Eq. Condition for a shedding vertex}, $ G[A\cup\{x\}]$ is connected. Furthermore,
		\begin{align*}
			\del_{\Sigma_1(A,G)}(x)&=\l V(G)\setminus (A\cup\{x\})\r\\
			&=\Delta_{V(G)\setminus (A\cup\{x\})}.
		\end{align*}
	\end{proof}
	\begin{remark}
		Note that if $A=\emptyset$, then in statement (ii) of (2) in \Cref{Eq. conditions for lk and del} the condition $x\in N_G(A)$ is no longer required, since $G[\{x\}]$ is always a connected graph for each $x\in V(G)$. In this case, we also have $\del_{\Sigma_1(\emptyset,G)}(x) =\Delta_{V(G)\setminus \{x\}}$
	\end{remark}
	We now proceed to show that $\SrAG$ is a vertex decomposable simplicial complex when $A$ is a non-empty subset of $V(G)$ such that $G[A]$ is connected. 

	\begin{theorem}\label{Vertex decoposable}
		Let $G$ be a graph, and let $A \subseteq V(G)$ be a non-empty set such that $G[A]$ is connected. Then for each $r \geq 1$, the simplicial complex $\Sigma_r(A,G)$ is vertex decomposable. 
	\end{theorem}
	\begin{proof}
		Let $|V(G)|=n$. We prove that $\SrAG$ is vertex decomposable using induction on $r+n$. By the given hypothesis $r\ge 1$, and since $A\neq\emptyset$, we must have $n\ge 1$.
		
		\medskip
		\noindent
		\textbf{Base case: }Let $r+n=2$. Thus, in this case, we have $|V(G)|=1$. Since $A \subseteq V(G)$ is non-empty, it must be the case that $A=V(G)$. In this setting, $\mathrm{Supp}_1(V(G),G)$ is the empty set. Therefore, the complex $\Sigma_1(A,G)$ is the void complex, which is trivially vertex decomposable.
		
		\medskip
		\noindent
		\textbf{Induction hypothesis: } Let $r,n$ be positive integers such that $r+n\ge 3$. We assume that if $r'$ and $n'$ are two positive integers such that $r'+n'<r+n$ and there exists a graph $G'$ with $|V(G')|=n'$, then $\Sigma_{r'}(A',G')$ is a vertex decomposable simplicial complex for any non-empty $A'\subseteq V(G')$ such that $G[A']$ is connected.
		
		\medskip
		\noindent
		\textbf{Inductive Step :} We divide this step into the following two cases:
		
		\medskip
		\noindent
		\textbf{Case I :} $r=1$. In this case, we have $\mathrm{Supp}_1(A,G) =N_G(A)$. Thus, \[
		\Sigma_1(A,G) = \langle V(G) \setminus (A \cup \{v\}) : v \in N_G(A) \rangle.
		\]
		Now, if |$N_G(A)| \le 1$, then $|\mathrm{Supp}_1(A,G)| \le 1$, and hence $\Sigma_1(A,G)$ is either the void complex or the empty complex, or a simplex, and hence a vertex decomposable simplicial complex. Next, consider the case when |$N_G(A)| \ge 2$. Fix any $x \in N_G(A)$. Observe that $x$ is a shedding vertex of $\Sigma_1(A,G)$. Indeed, for each $F\in \Supp_1(A,G)$ we have $|F|=1$. Fix such an $F$ with $x\notin F$. Suppose $F=\{y\}$. Then $(F\setminus\{y\})\cup\{x\}=\{x\}\in\Supp_1(A,G)$ since $G[A]$ is connected. Now, by \Cref{Eq. conditions for lk and del}, $\del_{\Sigma_1(A,G)}(x) =\Delta_{V(G)\setminus (A\cup\{x\})}$ is a simplex, and thus vertex decomposable. On the other hand, we have 
        \[\lk_{\Sigma_1(A,G)}(x) = \Sigma_1(A,G\setminus x),\] again by \Cref{Eq. conditions for lk and del}. Since $1+|V(G\setminus x)|<1+n$, using the induction hypothesis we see that $\lk_{\Sigma_1(A,G)}(x)$ is also vertex decomposable. Consequently, $\Sigma_1(A,G)$ is a vertex decomposable simplicial complex.
		
		\medskip
		\noindent
		\textbf{Case II :} $r\ge 2$. As before, if |$\Supp_r(A,G)| \le 1$, then $\SrAG$ is the void complex or the empty complex or a simplex, and thus trivially vertex decomposable. Therefore, we may assume that |$\Supp_r(A,G)| \ge 2$. Observe that for each $F\in\Supp_r(A,G)$, $F\cap N_G(A)\neq\emptyset$, and in particular, $N_G(A)\neq\emptyset$. Let $\Supp_r(A,G)=\{F_1,\ldots,F_k\}$. First consider the case when $N_G(A)\subseteq F_i$ for all $i\in[k]$. Thus $|N_G(A)|\le r$. If $|N_G(A)|=r$, then $|\Supp_r(A,G)|=1$, thus $\SrAG$ is a simplex, and hence a vertex decomposable complex. Now suppose $|N_G(A)|<r$. In this case, proceeding as in the proof of \Cref{Eq. conditions for lk and del} we observe that each $F\in\Supp_r(A,G)$ is in one-to-one correspondence with each $H\in \Supp_{r-|N_G(A)|}(A\cup N_G(A),G)$. Thus, 
		\[\langle F^{c} \setminus A : F \in \mathrm{Supp}_r(A,G) \rangle = \langle H^{c} \setminus (A \cup N_G(A)): H \in \mathrm{Supp}_{r-|N_G(A)|}(A \cup N_G(A)),G)\rangle \] 
		Hence, $\SrAG=\Sigma_{r-|N_G(A)|}(A\cup N_G(A),G)$. Since $N_G(A)\neq\emptyset$, by the induction hypothesis we have that $\SrAG$ is vertex decomposable.
		
		Next, we assume, without loss of generality, that $N_G(A)\not\subseteq F_1$. Choose some $x\in N_G(A)\setminus F_1$. Our first aim is to show that $x$ is a shedding vertex, and we show this by using \Cref{Eq. Condition for a shedding vertex}. Let $C$ be a connected component of $G[F]$. If $|C|=1$, and say $C=\{y\}$, then it is easy to see that 
        \[(F \setminus \{y\}) \cup \{x\} \in \mathrm{Supp}_r(A,G),\]
        since $G[A]$ and $G[F\cup A]$ both are connected. If $|C|\ge 2$, then by \Cref{connectivity lemma}, we see that $C$ has at least two vertices, say $y_1,y_2$ such that $G[C\setminus\{y_i\}]$ is connected for each $i\in\{1,2\}$. If $y_1\in N_G(A)$, then one can see that 
        \[
        (F \setminus \{y_2\}) \cup \{x\} \in \mathrm{Supp}_r(A,G),
        \]
        since $G[A]$ and $G[F\cup A]$ both are connected. On the other hand, if $y_1\notin N_G(A)$, then  $(F \setminus \{y_1\}) \cup \{x\} \in \mathrm{Supp}_r(A,G)$, by similar reason. Thus, $x$ is a shedding vertex of $\SrAG$. 
		
		Now, by \Cref{Eq. conditions for lk and del}, we have
		\[
		\lk_{\Sigma_r(A,G)}(x) = \Sigma_r(A,G \setminus x) \text{  and  }\del_{\Sigma_r(A,G)}(x) = \Sigma_{r-1}(A \cup \{x\},G).
		\] 
		Observe that $r+|V(G\setminus x)|<r+n$ and $r-1+|V(G)|<r+n$. Thus, by the induction hypothesis, we can say that both $\lk_{\Sigma_r(A, G)}(x)$ and $\del_{\Sigma_r(A,G)}(x)$ are vertex decomposable. Consequently, $\SrAG$ is a vertex decomposable simplicial complex, as desired. 
	\end{proof}

\begin{corollary}
    Let $G$ be a graph, and let $A \subseteq V(G)$ be a non-empty set such that $G[A]$ is connected. Then the following statements are equivalent:
    \begin{enumerate}
        \item $\Sigma_r(A,G)$ is vertex decomposable.
        \item $\Sigma_r(A,G)$ is shellable.
        \item $\Sigma_r(A,G)$ is Cohen-Macaulay.
        \item $\D_r(A,G)$ is chordal. 
    \end{enumerate}
\end{corollary}
\begin{proof}
    Directly follows from \Cref{chordal 2}, \Cref{Vertex decoposable} and \cite[Theorem 3.3]{SBCC}.
\end{proof}

\begin{remark}\label{betti_splitting}
    When $\Sigma_r(A,G)$ is vertex decomposable the dual ideal is vertex splittable, which means it admits a Betti splitting. 
    It follows from \cite[Theorem 2.3, Theorem 2.8, Remark 2.10]{MoradiAhang2016} that graded Betti numbers, projective dimension and regularity of the $r$-connected ideals can be computed recursively. 
    The remark also applies to the graph classes considered in \Cref{Section 4 Froberg}.
\end{remark}
    
	\begin{remark}\label{remark 1}
		If $A=\emptyset$, then it is easy to see that $\Supp_1(\emptyset, G)=V(G)$. Thus $\Sigma_1(\emptyset, G)$ is the boundary of a simplex, and hence vertex decomposable, by \cite[Proposition 2.2]{ProBill1980}. On the other hand, if $r\ge 2$, then one cannot proceed as in \Cref{Vertex decoposable}, as there is no such choice of shedding vertex in $x\in N_G(A)$. In this case, we require a different approach to determine whether $\Sigma_r(\emptyset,G)$ is vertex decomposable or not.
	\end{remark}
	
	Let $W\subseteq V(G)$ be such that $G[W]$ is connected, and $x\in V(G)\setminus W$. Then $W\cap N_G(x)\neq\emptyset$ if and only if there exists some $y\in W$ such that $G[(W\setminus\{y\})\cup\{x\}]$ is connected. Indeed, if $|W|=1$, then the ``only if'' part is trivially satisfied. Moreover, if $|W|\ge 2$, then using \Cref{connectivity lemma} we can find $y\in W$ such that $G[(W\setminus\{y\})\cup\{x\}]$ is connected. The ``if'' part follows directly from the definition. With this observation in mind, we prove the following.

	\begin{proposition}\label{Eq. condition for a shedding vertex of Sigma_r(G)}
		Let $G$ be a graph and $x \in V(G)$. The following are equivalent:
		\begin{enumerate}
			\item For every $\E \in E(\Con_r(G))$ with $x \notin \E$, we have $\E \cap N_G(x)\neq\emptyset$. In other words, for every $\E \in E(\Con_r(G))$ with $x \notin \E$, there exists a vertex $y \in \E$ such that $(\E \setminus \{y\}) \cup \{x\} \in E(\Con_r(G))$.
			\item $E(\Con_r(G \setminus N_G[x])) = \emptyset$.
		\end{enumerate}
		\begin{proof}
			\emph{(1) $\Rightarrow$ (2):} The proof proceeds by contradiction. Suppose that $E(\operatorname{Con}_r(G \setminus N_G[x])) \neq \emptyset$, and let $\mathcal{E} \in E(\operatorname{Con}_r(G \setminus N_G[x]))$. By construction, we then have $x\notin \mathcal{E}$ and $\mathcal{E} \cap N_G(x)=\emptyset$, which is a contradiction.

			\medskip
			\noindent
			\emph{(2) $\Rightarrow$ (1):} Suppose $E(\Con_r(G \setminus N_G[x])) = \emptyset$. Then for every $\E \in E(\Con_r(G))$, we have $\E \cap N_G[x] \neq \emptyset$. So, if $x \notin \E$, then we must have $\E \cap N_G(x) \neq \emptyset$, as required.
		\end{proof}
	\end{proposition} 
	In the next theorem, we provide an equivalent criterion for verifying the vertex decomposability of $\Sigma_r(\emptyset,G)=\Sigma_r(G)$.
    
\begin{theorem}\label{V.D of Sigma_r(G)}
		Let $G$ be a graph and $r\ge 2$ a positive integer. Then the following are equivalent:
		\begin{enumerate}
			\item Either $|E(\Con_r(G))|\le 1$ or there exists a positive integer $k$ and an ordering of the vertices $x_1,x_2,\ldots,x_k$ of $G$ such that:
			\begin{enumerate}
				\item For each $i\in [k-1]\cup\{0\}$,
				\[ E\left(\Con_r(G_i\setminus N_{G_i}[x_{i+1}])\right) = \emptyset,\]
				where $G_0=G$, and $G_i=G\setminus\{x_1,\ldots,x_i\}$ for each $i\ge 1$.
				\item $k$ is the least positive integer such that $|E(\Con_r(G\setminus\{x_1,\ldots,x_k\}))|\le 1$.
			\end{enumerate}
			\item The simplicial complex $\Sigma_r(G)$ is vertex decomposable.
		\end{enumerate}
	\end{theorem}
	\begin{proof}
		We have $\Sigma_r(G)=\l F^c\colon F\subseteq V(G), |F|=r, G[F]\text{ is connected}\r$. Also, by \Cref{Eq. Condition for a shedding vertex} and \Cref{Eq. condition for a shedding vertex of Sigma_r(G)}, $x\in V(G)$ is a shedding vertex of $\Sigma_r(G)$ if and only if $E(\Con_r(G \setminus N_G[x])) = \emptyset$. Building on these observations, we proceed to the next stage of the proof.
		
		\medskip
		\noindent
		\emph{(1)} $\Rightarrow$ \emph{(2)}: If $|E(\Con_r(G))|\le 1$, then $\Sigma_r(G)$ is either the void complex, the empty complex, or a simplex, and hence vertex decomposable. Thus, we may assume there exists a positive integer $k$ and an ordering of vertices $x_1,x_2,\ldots,x_k$ of $G$ which satisfies conditions (a) and (b). We have $E(\Con_r(G \setminus N_G[x_1])) = \emptyset$, and thus $x_1$ is a shedding vertex of $\Sigma_r(G)$. To show vertex decomposability, we must verify that both the deletion and the link of $x_1$ in $\Sigma_r(G)$ are vertex decomposable. By \Cref{Eq. conditions for lk and del} we have 
		\[
		\del_{\Sigma_r(G)}(x_1) = \Sigma_{r-1}(\{x_1\},G)\text{  and  }\lk_{\Sigma_r(G)}(x_1)=\Sigma_r(G \setminus x_1).
		\] 
		Now, by \Cref{Vertex decoposable}, $\del_{\Sigma_r(G)}(x_1)$ is vertex decomposable. Thus it remains to show that $\lk_{\Sigma_r(G)}(x_1)=\Sigma_r(G \setminus x_1)=\Sigma_r(G_1)$ is vertex decomposable. Since $E(\Con_r(G_1\setminus N_{G_1}[x_2]))=\emptyset$, we see that $x_2$ is a shedding vertex of $\Sigma_r(G \setminus x_1)$, and we continue the above process for $G_1$. After $k$-many steps, we finally obtain a vertex decomposable simplicial complex \[
        \lk_{\Sigma_r(G \setminus \{x_1, \ldots, x_{k-1}\})}(x_k) = \Sigma_r(G \setminus \{x_1, \ldots, x_k\}),
        \]
        which is either a simplex or the void complex, since $|E(\Con_r(G\setminus\{x_1,\ldots,x_k\}))| \le 1$. Consequently, $\Sigma_r(G)$ is vertex decomposable. 
		
		\medskip\noindent	
		\emph{(2)} $\Rightarrow$ \emph{(1)}: Assume that $\Sigma_r(G)$ is vertex decomposable. If $\Sigma_r(G)$ is a simplex or the void complex or the empty complex, then $|E(\Con_r(G))|\le 1$. Otherwise, there exists a shedding vertex, say $x_1$, such that both $\lk_{\Sigma_r(G)}(x_1)$ and $\del_{\Sigma_r(G)}(x_1)$ are vertex decomposable. Since $x_1$ is a shedding vertex we also have $E(\Con_r(G \setminus N_G[x_1])) = \emptyset$. If $|E(Con_r(G\setminus x_1))| \le 1$ then we are done. If not, then we have \[
        \lk_{\Sigma_r(G)}(x_1)=\Sigma_r(G \setminus x_1)=\Sigma_r(G_1),\]
        which is vertex decomposable. Thus it contains a shedding vertex, say $x_2$, for which $E(\Con_r(G_1 \setminus N_{G_1}[x_2])) = \emptyset$. Now, suppose $x_1, x_2, \ldots, x_{i-1}$ have already been chosen such that for each $j\in [i-1]$, |$E(\Con_r(G\setminus\{x_1,\ldots,x_{j}\}))|\ge 2$ and  $x_j$ is a shedding vertex of 
        \[
        \lk_{\Sigma_r(G_{j-2})}(x_{j-1})=\Sigma_r(G_{j-2} \setminus x_{j-1})=\Sigma_r(G_{j-1}).
        \]
        Thus, $E(\Con_r(G_{i-1} \setminus N_{G_{i-1}}[x_i])) = \emptyset$, and hence, we have part (a) of (1). Moreover, since $|V(G)|<\infty$ after $k$-many steps we have $\lk_{\Sigma_r(G \setminus \{x_1, \ldots, x_{k-1}\})}(x_k) = \Sigma_r(G \setminus \{x_1, \ldots, x_k\})$ is either a simplex or the empty complex or the void complex. Thus, $|E(\Con_r(G\setminus\{x_1,\ldots,x_k\})) | \le 1$, which is part (b) of (1). 
	\end{proof}

In \cite[Theorem 3.10]{CCVD}, the author has proved that if the Alexander dual of the independence complex of $\H^{c}$ is vertex decomposable, then $\H$ is chordal. In case of $\Con_r(G)$, we can provide an independent proof of this fact using \Cref{chordal 4} and \Cref{V.D of Sigma_r(G)}. Recall that $(\Con_r(G))^{c} = \D_r(G)$ for any graph $G$, and thus, we have the following theorem. 

\begin{theorem}\label{V.D implies chordality}
    If the simplicial complex $\Sigma_r(G)$ is vertex decomposable, then the clutter $\D_r(G)$ is chordal. 
\end{theorem}
\begin{proof}
   From \Cref{V.D of Sigma_r(G)}, we obtain an equivalent condition for the vertex decomposability of $\Sigma_r(G)$, and from \Cref{chordal 4} we obtain a sufficient condition for $\D_r(G)$ to be chordal. It is therefore enough to show that condition~(1) in \Cref{V.D of Sigma_r(G)} implies condition~(1) in \Cref{chordal 4}. 

If $k$ is the smallest positive integer such that
\[
    |E(\Con_r(G \setminus \{x_1,\ldots,x_k\}))| \le 1,
\]
then the clutter $\D_r(G \setminus \{x_1,\ldots,x_k\})$ is either the $r$-uniform complete clutter on $V(G) \setminus \{x_1,\ldots,x_k\}$ or the clutter obtained by deleting a single circuit from the $r$-uniform complete clutter on $V(G) \setminus \{x_1,\ldots,x_k\}$. In both cases, $\D_r(G)$ is chordal.

\end{proof}

\medskip
    
Our next goal is to establish necessary and sufficient conditions for determining when the complexes $\Sigma_r(G_1 \sqcup G_2)$ and $\Sigma_r(G_1 * G_2)$ are vertex decomposable, under the assumption that $V(G_1) \cap V(G_2) = \emptyset$. We first recall the constructions of $G_1 \sqcup G_2$ and $G_1 * G_2$. Let $G_1$ and $G_2$ be two graphs on disjoint vertex sets $V(G_1)$ and $V(G_2)$. The {\it disjoint union} $G_1\sqcup G_2$ of $G_1$ and $G_2$ is a graph with vertex set and circuit set described as follows:
    \begin{align*}
       V(G_1\sqcup G_2)&=V(G_1)\cup V(G_2)\\ 
       E(G_1\sqcup G_2)&=E(G_1)\cup E(G_2).
    \end{align*}
    Similarly, the {\it join} of $G_1$ and $G_2$, denoted by $G_1*G_2$, is a graph with the vertex set and circuit set described as follows:
    \begin{align*}
        V(G_1*G_2)&=V(G_1)\cup V(G_2)\\
        E(G_1*G_2)&=E(G_1)\cup E(G_2)\cup\{\{x,y\}\mid x\in V(G_1)\text{ and }y\in V(G_2)\}.
    \end{align*}
    
    The following two lemmas play a crucial role in formulating the vertex decomposable conditions of $\Sigma_r(G_1 \sqcup G_2)$ and $\Sigma_r(G_1 * G_2)$.

	\begin{lemma}\label{Sigma of the disjoint union of graphs}
		Let $G_1$ and $G_2$ be two disjoint graphs. Then
		\[
		\Sigma_r(G_1 \sqcup G_2)
		= (\Sigma_r(G_1) * \Delta_{V(G_2)})
		\;\sqcup\;
		(\Delta_{V(G_1)} * \Sigma_r(G_2)).
		\]
	\end{lemma}
	
	\begin{proof}
    This follows directly from the fact that $\mathcal{E} \in E(\operatorname{Con}_r(G_1 \sqcup G_2))$ if and only if either $\mathcal{E} \in E(\operatorname{Con}_r(G_1))$ or $\mathcal{E} \in E(\operatorname{Con}_r(G_2))$.\end{proof}
	
	\begin{lemma}\label{Sigma of the join of graphs}
		Let $G_1$ and $G_2$ be two disjoint graphs. Then
		\begin{align*}
		    \Sigma_r(G_1 * G_2)
		= & \Sigma_r(G_1) * \langle V(G_2) \rangle
		\;\sqcup\;
		   \langle V(G_1) \rangle * \Sigma_r(G_2)\\
		& \sqcup\; 
		\big\langle\, \E^c \;:\; |\E|=r,\; \E \cap V(G_1)\neq\emptyset,\; \E \cap V(G_2)\neq\emptyset \,\big\rangle.
		\end{align*}
	\end{lemma}
	
	\begin{proof}
		Let $W$ be a facet of $\Sigma_r(G_1 * G_2)$. Then $W=\E^c$ for some $\E \in E(\Con_r(G_1*G_2))$. Thus, in order to determine the facets of $\Sigma_r(G_1 * G_2)$, we need to determine the structure of the circuits in the clutter $\Con_r(G_1*G_2)$. 
        Let $\E \in E(\Con_r(G_1*G_2))$. Then there are three possibilities:
		
		\begin{itemize}
			\item $\E \subseteq V(G_1)$. Hence, $		\E^c=(V(G_1)\setminus \E)\;\sqcup\; V(G_2)\in \Sigma_r(G_1)*\langle V(G_2)\rangle.$
			
			\item $\E \subseteq V(G_2)$. In this case, $\E^c=V(G_1)\;\sqcup\; (V(G_2)\setminus \E)\in \langle V(G_1)\rangle* \Sigma_r(G_2)$.
			
			\item $\E\cap V(G_1)\neq\emptyset$ and $\E\cap V(G_2)\neq\emptyset$. In this case, it is easy to see that \[\E^c\in \big\langle\, \E^c \;:\; |\E|=r,\; \E \cap V(G_1)\neq\emptyset,\; \E \cap V(G_2)\neq\emptyset \,\big\rangle. 
			\]
		\end{itemize}
        This completes the proof.
\end{proof}

Based on the above two lemmas, we now proceed to prove the following two theorems.

	\begin{theorem}\label{disjoint union vd}
		Let $r \geq 2$ and let $G_1$ and $G_2$ be graphs.  
		Then $\Sigma_r(G_1 \sqcup G_2)$ is vertex decomposable if and only if one of the following two conditions holds:
		\begin{itemize}
			\item $\Sigma_r(G_1)$ is vertex decomposable and $|V(G_2)| < r$.
			
			\item $\Sigma_r(G_2)$ is vertex decomposable and $|V(G_1)| < r$.
		\end{itemize}  
	\end{theorem}
	
	\begin{proof}
		Observe that if $|V(G_1)| \geq r$ and $|V(G_2)| \geq r$, then by \Cref{Eq. condition for a shedding vertex of Sigma_r(G)},  
		$\Sigma_r(G_1 \sqcup G_2)$ has no shedding vertex and is therefore not vertex decomposable.  
		Thus, if $\Sigma_r(G_1 \sqcup G_2)$ is vertex decomposable, at least one of the two graphs must have fewer than $r$ vertices.  
		Without loss of generality, assume $|V(G_2)| < r$. Then $\Sigma_r(G_2)$ is the empty complex, and hence by \Cref{Sigma of the disjoint union of graphs}, we have
		\[
		\Sigma_r(G_1 \sqcup G_2) 
		= \Sigma_r(G_1) * \Delta_{V(G_2)}.
		\]
		Now, by \cite[Proposition 2.4]{ProBill1980}, $\Sigma_r(G_1) * \Delta_{V(G_2)}$ is vertex decomposable if and only if $\Sigma_r(G_1)$ is vertex decomposable. Thus, we have $\Sigma_r(G_1)$ is vertex decomposable, as required.
		
		Conversely, suppose $\Sigma_r(G_1)$ is vertex decomposable and $|V(G_2)| < r$, then  
		$\Sigma_r(G_1 \sqcup G_2) = \Sigma_r(G_1) * \langle V(G_2) \rangle$ is vertex decomposable, again by \cite[Proposition 2.4]{ProBill1980}).
	\end{proof}

	\begin{theorem}\label{V.D. of the join of graphs}
		Let $r\ge 2$, and let $G_1,G_2$ be two graphs. Then $\Sigma_r(G_1 * G_2)$ is vertex decomposable if and only if both $\Sigma_r(G_1)$ and $\Sigma_r(G_2)$ are vertex decomposable.
	\end{theorem}
	
	\begin{proof}
		Assume that $\Sigma_r(G_1*G_2)$ is vertex decomposable. If $\Sigma_r(G_1) = \emptyset$, then it is vertex decomposable. If not then, by \Cref{Sigma of the join of graphs}, we have $\lk_{\Sigma_r(G_1*G_2)}(V(G_2))=\Sigma_r(G_1)$. Since the link of any face of a vertex decomposable simplicial complex is vertex decomposable (see \cite[Proposition 2.3]{ProBill1980}), we find that $\Sigma_r(G_1)$ is vertex decomposable. Similarly, considering the link of the face $V(G_1)$, we find that $\Sigma_r(G_2)$ is vertex decomposable.
		
		Conversely, assume that both $\Sigma_r(G_1)$ and $\Sigma_r(G_2)$ are vertex decomposable. Let $G=G_1*G_2$ with $|V(G_1)|=n_1$ and $|V(G_2)| = n_2$ . We prove that $\Sigma_r(G)$ is vertex decomposable using induction on $|V(G)|=n_1+n_2$.
        
			\medskip
			\noindent
			\textbf{Base case:} $n_1+n_2 \le r$. Then $\Sigma_r(G)$ is either the void complex or the empty complex, which is clearly vertex decomposable.
			
            \medskip
			\noindent
			\textbf{Induction hypothesis:} Let $n_1$ and $n_2$ be positive integers such that $n_1 + n_2 \ge r+1$. We assume that if $n_1'$ and $n_2'$ are two positive integers such that $n_1' + n_2' < n_1 + n_2$ and there exists graphs $G_1'$ and $G_2'$ with $|V(G_1')| = n_1'$ and $|V(G_2')| = n_2'$, then whenever $\Sigma_r(G_1')$ and $\Sigma_r(G_2')$ are vertex decomposable, so is $\Sigma_r(G_1' * G_2')$.
			
            \medskip
			\noindent
			\textbf{Inductive Step:} We choose a shedding vertex $x$ of $\Sigma_r(G)$ as follows: if $\Sigma_r(G_1)$ is a simplex or the empty complex or the void complex, then take $x$ to be any vertex of $G_1$. Otherwise, take $x\in V(G_1)$ such that $x$ is a shedding vertex of $\Sigma_r(G_1)$. Then, by \Cref{Eq. Condition for a shedding vertex}, $x$ is a shedding vertex of $\Sigma_r(G)$. Indeed,
            \[E(\Con_r(G\setminus N_G[x])) = E(\Con_r(G_1 \setminus N_{G_1}[x]) = \emptyset.\]
            Next, by \Cref{Eq. conditions for lk and del}, we have
		\[
		\lk_{\Sigma_r(G)}(x) = \Sigma_r(G \setminus x) \text{  and  }\del_{\Sigma_r(G)}(x) = \Sigma_{r-1}(\{x\},G).
		\]
Observe that $ G\setminus x = (G_1 \setminus x)*G_2$ and $\lk_{\Sigma_r(G_1)}(x) = \Sigma_r(G_1 \setminus x)$ is vertex decomposable. Since, $|V((G_1\setminus x)*G_2)| = n_1-1+n_2 < n_1 + n_2$, by the induction hypothesis $\Sigma_r(G\setminus x)$ is vertex decomposable. Thus, by \Cref{Vertex decoposable}, $\Sigma_r(G)$, is vertex decomposable.
            \end{proof}

	\begin{remark}
		Note that in case $r=1$, by \Cref{remark 1}, we always have that both the simplicial complexes $\Sigma_1(G_1\sqcup G_2)$ and $\Sigma_1(G_1*G_2)$ are vertex decomposable for any choice of $G_1$ and $G_2$.
	\end{remark}

	
\section{Generalizing Fr\"oberg's theorem to Co-Connected Complexes}\label{Section 4 Froberg}
	
In this section, we apply the results established in \Cref{sec-3} to investigate the vertex decomposability, shellability, and Cohen-Macaulay property of the simplicial complex $\SrAG$ for several important classes of graphs.
As an application, we obtain a clutter version of the well-known Fr\"oberg's theorem for $\Con_r(G)$; we start with the chordal graphs. 
	
	\begin{theorem}\label{V.D of Sigma_r(G) if G is chordal}
		Let $G$ be a chordal graph, and $r \geq 2$. The following are equivalent:
		\begin{enumerate}        
			\item $\Sigma_r(G)$ is vertex decomposable.
			\item $\Sigma_r(G)$ is shellable.
			\item $\Sigma_r(G)$ is Cohen-Macaulay.
			\item $G$ is $r$-gap-free.
			\item The clutter $\Con_r(G)$ is co-chordal.
		\end{enumerate}
	\end{theorem}
	\begin{proof}
		The implications \emph{(1) $\Rightarrow$ (2) $\Rightarrow$(3)} follows from the known hierarchy of conditions in \Cref{VdShCm}. Also, \emph{(3)$ \Rightarrow$ (4)} follows from \cite[Theorem 3]{EagonReiner1998} and \cite[Theorem 1.4]{HaWoodroofe2014}, and \emph{(1) $\Rightarrow$ (5)} follows from \Cref{V.D implies chordality}. Moreover, \emph{(5) $\Rightarrow$ (3)} follows from \cite[Theorem 3.3]{SBCC} and \cite[Theorem 3]{EagonReiner1998}. Thus, to conclude the proof of the theorem, it remains to show \emph{(4)} $\Rightarrow$ \emph{(1)}, and we prove this by induction on $|V(G)|$.  

\medskip
        \noindent
        {\bf Base Case:} $|V(G)| \leq r $. Recall that $\Sigma_{r}(G)=\l V(G)\setminus F\mid F\in E(\Con_r(G))\r$. Thus, in this case, $\Sigma_r(G) $ is either the void complex or the empty complex, and hence a vertex decomposable simplicial complex.
        
        \medskip
        \noindent
        {\bf Induction hypothesis:} Let $|V(G)|\ge r+1$, and if $\widehat G$ is a chordal graph such that $|V(\widehat G)|<|V(G)|$ and $\widehat G$ is $r$-gap-free, then $\Sigma_r(\widehat G)$ is vertex decomposable.

\medskip
        \noindent
        {\bf Inductive Step:} Since $G$ is $r$-gap-free we have $\gamma_r(G) \leq 1$. If $\gamma_r(G) = 0$, then $G = G_1 \bigsqcup G_2 \bigsqcup \cdots \bigsqcup G_l$, where $|V(G_i)| \leq r - 1$ for each $i\in[l]$. Therefore, by \Cref{disjoint union vd} it follows that $\Sigma_r(G)$ is vertex decomposable since each $\Sigma_{r}(G_i)$ is a void complex.
		
		Next, we consider the case when $\gamma_r(G) = 1$. If $|E(\Con_r(G))|=1$, then $\Sigma_{r}(G)$ is a simplex, and thus vertex decomposable. Hence, we may assume that $|E(\Con_r(G))|\ge 2$. Note that to show $\Sigma_r(G)$ is vertex decomposable, it is enough to find a shedding vertex. Indeed, if $x$ is a shedding vertex of $\Sigma_{r}(G)$, then by \Cref{Eq. conditions for lk and del} we have
		\[
		\lk_{\Sigma_{r}(G)}(x)=\Sigma_r(G\setminus x), \text{ and } \del_{\Sigma_{r}(G)}(x)=\Sigma_{r-1}(\{x\},G).
		\]
		Since $G\setminus x$ is a chordal graph and $\gamma_r(G\setminus x)\le\gamma_r(G)\le 1$, using the induction hypothesis and by \Cref{Vertex decoposable}, we obtain that the simplicial complexes $\Sigma_r(G\setminus x)$ and $\Sigma_{r-1}(\{x\},G)$ both are vertex decomposable. Thus $\Sigma_{r}(G)$ is a vertex decomposable simplicial complex.
		
		To find a shedding vertex of $\Sigma_{r}(G)$, we use the combinatorial description available in \Cref{Eq. condition for a shedding vertex of Sigma_r(G)}. More specifically, our aim now is to find a vertex $x \in V(G)$ such that
		$E(\Con_r(G \setminus N_G[x])) = \emptyset$. Note that, since $G$ is chordal, $G$ contains a simplicial vertex, say $v$. Let $G'$ denote the graph $G\setminus v$. As before, by the induction hypothesis, we have that $\Sigma_{r}(G')$ is vertex decomposable and hence, it contains a shedding vertex, say $w\in V(G)\setminus\{v\}$. Therefore, by \Cref{Eq. condition for a shedding vertex of Sigma_r(G)} we have $E(\Con_r(G' \setminus N_{G'}[w])) = \emptyset$. We have the following two cases to consider.
		
		\medskip
		\noindent
		{\bf Case I: } $v\in N_G[w]$. In this case, it is easy to see that $E(\Con_r\left(G \setminus N_G[w]\right)) = \emptyset$, and thus $w$ is a shedding vertex of $\Sigma_{r}(G)$. 
		
		\medskip
		\noindent
		{\bf Case II: } $v\notin N_G[w]$. Since $E(\Con_r(G' \setminus N_{G'}[w])) = \emptyset$, in this case we can write $G \setminus (N_G[w] \cup \{v\}) = G_1 \sqcup G_2 \sqcup \cdots \sqcup G_k$, where $|V(G_i)| \leq r - 1$ for each $i=1,\ldots,k$. Moreover, since $v$ is a simplicial vertex of $G$, we can assume, without loss of generality, that $N_G(v) \setminus N_G(w) \subseteq V(G_1)$. Let $G_1'=G[V(G_1)\cup\{v\}]$. Then we can write 
        \[
        G \setminus N_G[w] = G_1' \sqcup G_2 \sqcup \cdots \sqcup G_k,
        \]
        where $|V(G_i)| \leq r - 1$ for $i = 2,\ldots,k$ and $|V(G_1')| \leq r$.
		
		Observe that if $|V(G_1')|\le r-1$, then $E(\Con_r\left(G \setminus N_G[w]\right)) = \emptyset$, and thus $w$ is a shedding vertex of $\Sigma_{r}(G)$. Thus, we may assume that $|V(G_1')|=r$. In this case, let \[
        E(\Con_r(G))=\{\E,\E_1,\ldots,\E_p\},
        \]where $\E=\{u\mid u\in V(G_1')\}$. 
		Note that since $\gamma_r(G)\le 1$, for each $i\in [p]$ there exists some $x_i\in\E$ and $w_i\in \E_i$ such that $\{x_i,w_i\}\in E(G)$. Since $G_1'$ and $G_i$ are disconnected for each $i\in [k]$, we must have $w_i\in N_G(w)$. It may happen that for some $i\neq j$, $w_i=w_j$. Thus, without loss of generality, we may assume that there exists some integer $s\le p$ such that for each $i\in [s]$ we have $x_i\in \E$ and distinct $w_i\in N_G(w)\cap \E_i$ such that $\{x_i,w_i\}\in E(G)$ and furthermore, for each $i\in [p]\setminus[s]$ $\E_i\cap\{w_1,\ldots,w_s\}\neq\emptyset$. Let $W=\{w_1,\ldots,w_s\}$. Our aim now is to prove the following.
		
		\medskip
		\noindent
		{\bf Claim}: $G[W]$ forms a clique in $G$.
		
		\medskip
		\noindent
		{\bf Proof of Claim:} Let us choose some distinct $i,j\in [s]$. Since $G_1'$ forms a connected component of $G$, let us choose an induced path $\mathcal P$ from $x_i$ to $x_j$. Then $\mathcal P\cup\{w_i,w,w_j\}$ forms a path in $G$. By construction, there is no circuit between the vertices of $\mathcal P$ and $w$. However, there might be circuits between the vertices of $\mathcal P$ and $w_i,w_j$. Thus, if necessary, considering an induced path between $w_i$ and $w_j$ via the vertices of $\mathcal P$, we obtain a cycle of length at least $4$ on the vertices $w,w_i,w_j$ and a subset of the vertices of $\mathcal P$. Since $G$ is chordal, we must have $\{w_i,w_j\}$, and this completes the proof of the claim. 
		
		Using the claim we see that for each $i\in [s]$, $N_G[w_i]\cap\E\neq\emptyset$ and $N_G[w_i]\cap \E_i\neq\emptyset$ for all $i\in [p]$.Consequently, for each $i\in[s]$, $E(\Con_r(G \setminus N_G[w_i])) = \emptyset$, and hence, $w_i$ is a shedding vertex of $\Sigma_r(G)$. This completes the proof of the theorem.
	\end{proof}
	
	\begin{remark}
		In \cite{LQCG}, it was shown that if $G$ is a chordal graph, then for the $r$-connected ideal, the linear resolution and linear quotient property both are equivalent to $G$ being $r$-gap-free. Recall from \Cref{preliminaries} that the $r$-connected ideal of $G$ is the Stanley-Reisner ideal of $\Indr$. Thus, using \cite[Theorem 2.3]{MoradiAhang2016}, we obtain an extension of \cite[Theorem 5.1]{LQCG}: for the $r$-connected ideal of a chordal graph $G$, the linear resolution, linear quotient, and the vertex splittable property are all equivalent to $G$ being $r$-gap-free. This also provides an affirmative answer to the first part of \cite[Question 7.5]{LQCG}.
	\end{remark}
	
	The next theorem concerns the complex $\Sigma_r(G)$ in the case where $G$ is a co-chordal graph. We note that an equivalent algebraic version of this result was previously established in \cite[Theorem 3.12]{FVDH}, using the concept of vertex splittability of a monomial ideal. In contrast, our proof is combinatorial in nature and relies on the sufficient condition for vertex decomposability of $\Sigma_r(G)$ developed in \Cref{sec-3}. 

    
    
	\begin{theorem}\label{V.D of Sigma_r(G) if G is co-chordal}
		For $r\geq 2$, if $G$ is co-chordal, then $\Sigma_r(G)$ is vertex decomposable. 
	\end{theorem}
	\begin{proof}
		The proof is again by induction on $|V(G)|$. As before, the base case of the induction procedure is easy to see since $|V(G)| \leq r $ implies $\Sigma_r(G) $ is either the void complex or the empty complex. Thus, we may assume $|V(G)|>r$. Note that for each $y\in V(G)$, the induced subgraph $G\setminus y$ is again a co-chordal graph. Thus, by the induction hypothesis, \Cref{Eq. conditions for lk and del}, and \Cref{Vertex decoposable} it follows that 
        \[
        \lk_{\Sigma_{r}(G)}(y)=\Sigma_r(G\setminus y), \text{ and } \del_{\Sigma_{r}(G)}(y)=\Sigma_{r-1}(\{y\},G),\] both are vertex decomposable simplicial complex. Consequently, $\Sigma_r(G)$ is vertex decomposable provided that $y$ is a shedding vertex of $\Sigma_r(G)$.
		
		To find a shedding vertex, we use \Cref{Eq. condition for a shedding vertex of Sigma_r(G)} and the fact that the complement of $G$ is a chordal graph. Indeed, let $x$ be a simplicial vertex of $G^c$. Then, one can observe that $E(G\setminus N_G[x])=\emptyset$. In particular, $E(\Con_r(G\setminus N_G[x])) = \emptyset$ since $r\ge 2$. Thus, by \Cref{Eq. condition for a shedding vertex of Sigma_r(G)} $x$ is a shedding vertex of $\Sigma_r(G)$, and hence, $\Sigma_r(G)$ is vertex decomposable, as desired.
	\end{proof}
\begin{remark}\label{coChordal123}
    It follows from \Cref{V.D of Sigma_r(G) if G is co-chordal} and \Cref{V.D implies chordality} that if $G$ is co-chordal, then $\Con_r(G)$ is a co-chordal clutter, which shows that \cite[Conjecture 6.1]{FVDH} is true. Here, we remark that one can also deduce \cite[Conjecture 6.1]{FVDH} in the affirmative by using \cite[Theorem 3.12]{FVDH}, \cite[Theorem 2.3]{MoradiAhang2016}, and \Cref{V.D implies chordality}.
\end{remark}
    
	\begin{remark}
		Since $G$ is co-chordal, we have $\gamma_1(G)\le 1$, and hence $\gamma_r(G)\le 1$. Thus, it is evident that for a co-chordal graph $G$ and for each $r\ge 2$, we have $G$ is $r$-gap-free, and the complex $\Sigma_{r}(G)$ is vertex decomposable, and hence, shellable, and Cohen-Macaulay. 
	\end{remark}

 An analogous result to those established in the preceding theorems also holds for cographs. 

\begin{definition}
A \emph{cograph} is defined recursively as follows:

\begin{enumerate}
    \item \textbf{Base Case:}\\
    Every single-vertex graph $K_1$ is a cograph.

    \item \textbf{Recursive Step:} \\
    If $G_1$ and $G_2$ are cographs, then the following graphs are also cographs:
    \begin{itemize}
        \item \textbf{Disjoint union:} $G = G_1 \sqcup G_2$.
        \item \textbf{Join:} $G = G_1 * G_2$.
    \end{itemize}

    \item \textbf{Closure:} \\
    Every cograph can be obtained from $K_1$ by a finite sequence of the above two operations.
\end{enumerate}
\end{definition}

The structural simplicity of cographs, arising from their recursive construction via disjoint unions and joins, allows us to analyze the equivalence effectively.

\begin{theorem}\label{V.D of Sigma_r(G) if G is cograph}
		Let $G$ be a cograph, and $r \geq 2$. The following are equivalent:
		\begin{enumerate}        
			\item $\Sigma_r(G)$ is vertex decomposable.
			\item $\Sigma_r(G)$ is shellable.
			\item $\Sigma_r(G)$ is Cohen-Macaulay.
			\item $G$ is $r$-gap-free.
			\item The clutter $\Con_r(G)$ is co-chordal.
		\end{enumerate}
	\end{theorem}
\begin{proof}
We proceed similarly to \Cref{V.D of Sigma_r(G) if G is chordal}. The implications \emph{(1) $\Rightarrow$ (2) $\Rightarrow$ (3)} follow from the known hierarchy of conditions in \Cref{VdShCm}. Furthermore, \emph{(3) $\Rightarrow$ (4)} follows from \cite[Theorem 3]{EagonReiner1998} and \cite[Theorem 1.4]{HaWoodroofe2014}, while \emph{(1) $\Rightarrow$ (5)} follows from \Cref{V.D implies chordality}. Finally, \emph{(5) $\Rightarrow$ (3)} follows from \cite[Theorem 3.3]{SBCC} and \cite[Theorem 3]{EagonReiner1998}. Hence, to complete the proof, it remains to establish \emph{(4) $\Rightarrow$ (1)}, which we do by induction on $|V(G)|$.

\medskip
\noindent
{\bf Base Case:} If $|V(G)| \le r$, then $\Sigma_r(G)$ is either the void complex or the empty complex, and is therefore vertex decomposable.

\medskip
\noindent
{\bf Induction Hypothesis:} Let $|V(G)| \ge r+1$, and assume that for any $r$-gap-free cograph $\widehat G$ with $|V(\widehat G)| < |V(G)|$, the complex $\Sigma_r(\widehat G)$ is vertex decomposable.

\medskip
\noindent
{\bf Inductive Step:} Since $G$ is an $r$-gap-free cograph, the recursive definition of cographs yields the following two cases:

\medskip
\noindent
\textbf{Case 1.} 

Suppose that $G$ is the disjoint union of two nonempty cographs $G_1$ and $G_2$. Write
\[
G_1 = \bigsqcup_{i=1}^{n_1} G_{1i}
\qquad\text{and}\qquad
G_2 = \bigsqcup_{j=1}^{n_2} G_{2j},
\]
where $G_{ij}$'s are connected components of $G_i$ for each $i\in[2]$ and $j\in [n_i]$. Since $G$ is $r$-gap-free, either every connected component of $G$ has vertex-set cardinality at most $r-1$, or there is exactly one connected component whose vertex-set cardinality is at least $r$. Without loss of generality, assume that if such a component exists, then it is $G_{2n_2}$.

Now consider
\[
G_1' = \left( \bigsqcup_{i=1}^{n_1} G_{1i} \right) \sqcup \left( \bigsqcup_{j=1}^{n_2-1} G_{2j} \right).
\]
Then
\[
G = G_1' \sqcup G_{2n_2}.
\]

As $G$ is a cograph, $G_1'$ and $G_{2n_2}$ are cographs. Moreover, both $G_1'$ and $G_{2n_2}$ are $r$-gap-free. Thus, by the induction hypothesis and by \Cref{disjoint union vd}, it follows that $\Sigma_r(G)$ is vertex decomposable.

\medskip
\noindent
{\bf Case 2:} $G$ is the join of two nonempty cographs $G_1$ and $G_2$. Observe that both $G_1$ and $G_2$ are $r$-gap-free, and thus, by the induction hypothesis, $\Sigma_r(G_1)$ and $\Sigma_r(G_2)$ are vertex decomposable. Hence, by \Cref{V.D. of the join of graphs}, we have that $\Sigma_r(G)$ is vertex decomposable.

\noindent
This completes the inductive step, and hence the proof.
\end{proof}

\begin{remark}
    Complete multipartite graphs form a subclass of cographs. Hence, by \Cref{V.D of Sigma_r(G) if G is cograph}, for a complete multipartite graph $G$, the properties of vertex decomposability, shellability, and Cohen-Macaulayness of $\Sigma_r(G)$ are all equivalent to $G$ being $r$-gap-free. Furthermore, this is equivalent to the associated clutter $\mathrm{Con}_r(G)$ being co-chordal.

\end{remark}
    
We now turn our attention to the complex $\Sigma_r(C_n)$. We begin with the following lemma, whose proof depends on the computation of the Leray number of $\Ind_r(C_n)$. Recall that, a simplicial complex $\Delta$ is called {\it $d$-Leray} if $\tilde{H}_i(Y;\mathbb{Z})=0$ for all induced subcomplexes $Y \subseteq \Delta$ and for all $i\geq d$. The {\it Leray number} of $\Delta$, denoted $L(\Delta)$, is the minimal $d$ such that $\Delta$ is $d$-Leray.
    \begin{lemma}\label{cycle lemma}
        For the cycle graph $C_n$, if the complex $\Sigma_r(C_n)$ is Cohen-Macaulay, then $n\le r+2$.
    \end{lemma}
    \begin{proof}
       We show that if $n>r+2$, then $\Sigma_r(C_n)$ is not Cohen-Macaulay. Indeed, by the Eagon-Reiner theorem \cite[Theorem 3]{EagonReiner1998} and Hochster's formula \cite{Hochster77}, it is enough to show that the Leray number $L(\Ind_{r}(C_n))>r-1$ for each $n>r+2$. By \cite[Theorem 4.6]{HICH} we have 
       \begin{align}\label{eq3546}
       \Ind_{r}(C_n) \cong
\begin{cases} 
     \displaystyle \bigvee_{r}\mathbb{S}^{rk-k-1} & \mathrm{if\ } n =(r+1)k;\\
     \mathbb{S}^{rk-k-1} & \mathrm{if\ } n = (r+1)k+ 1;\\
     \mathbb{S}^{rk-k} & \mathrm{if\ } n = (r+1)k+2;\\
     \vdots & \vdots\\
     \mathbb{S}^{rk-k+r-2} & \mathrm{if\ } n = (r+1)k+r,
   \end{cases}
\end{align}
where $n-1\geq r \geq 2$ and $k=\lfloor\frac{n}{r+1}\rfloor$. Since $n>r+2$, we can write $n=(r+1)k+l$, where $2\le l\le r$ in case $k=1$, and $0\le l\le r$ in case $k\ge 2$. First, consider the case when $2\le l\le r$ with $k\ge 1$. Then from \Cref{eq3546} we have 
\[
\tilde{H}_{rk-k+l-2}(\Ind_{r}(C_n);\mathbb{K}) \cong \mathbb Z.
\] 
Observe that $rk-k+l-2\ge r-1$ since $l\ge 2$ $k\ge 1$, and $r\ge 2$. Thus, $L(\Ind_r(C_n)>r-1$, as required. Next, we consider the case where $l=0,1$ and $k\ge 2$. In this case, again by \Cref{eq3546}, we have
\begin{align*}
\tilde{H}_{rk-k-1}(\Ind_{r}(C_n);\mathbb{K}) \cong \begin{cases}
    \mathbb Z^r &\text{ if }l=0,\\
    \mathbb Z &\text{ if }l=1.
\end{cases} 
\end{align*}
Observe that $rk-k-1\ge r-1$ since $k\ge 2$ and $r\ge 2$. Thus, $L(\Ind_r(C_n)>r-1$ in this case too, and this completes the proof.
    \end{proof}
    \begin{remark}
        The conclusion of \Cref{cycle lemma} can also be deduced from \cite[Proposition 5.3]{LQCG}, where the authors have used the fact that for cycle graphs, the $r$-connected ideal is the same as the $r$-path ideal.
    \end{remark}
        
	\begin{theorem}\label{cycle vd theorem} 
		Let $n \geq 3$ and $r\geq 2$, and let $C_n$ be the cycle graph on $n$ vertices. The following statements are equivalent:
		\begin{enumerate}        
			\item $\Sigma_r(C_n)$ is vertex decomposable.
			\item $\Sigma_r(C_n)$ is shellable.
			\item $\Sigma_r(C_n)$ is Cohen-Macaulay.
			\item $n \leq r+2$.
			\item The clutter $\Con_r(C_n)$ is co-chordal.
		\end{enumerate}
	\end{theorem}
	\begin{proof}
		We first show that \emph{(4) $\Rightarrow$ (1)}. As before, by \Cref{Eq. conditions for lk and del}, for any $x\in V(C_n)$ we have  
        \[
        \del_{\Sigma_r(C_n)}(x) = \Sigma_{r-1}(\{x\},C_n),
        \]
        which is vertex decomposable (by \Cref{Vertex decoposable}). Furthermore, 
    \[\lk_{\Sigma_r(C_n)}(x) = \Sigma_r(C_n\setminus x),
    \]
    where $C_n \setminus x$ is a path on $n - 1 \leq r + 1$ vertices, which is both $r$-gap-free and chordal. Therefore, by \Cref{V.D of Sigma_r(G) if G is chordal}, $\Sigma_r(C_n \setminus x)$ is vertex decomposable. Thus, in order to show that $\Sigma_r(C_n)$ is vertex decomposable, we need to find a shedding vertex. Note that for any vertex $x \in V(C_n)$, the graph $C_n \setminus N_{C_n}[x]$ is a path on $n-3$ vertices. Since $n \leq r+2$, we have
		$|V(C_n \setminus N_{C_n}[x])|\leq r - 1$, and hence, $E(\Con_r(C_n \setminus N_{C_n}[x])) = \emptyset$. Therefore, any $x\in V(C_n)$ is a shedding vertex of $\Sigma_{r}(C_n)$. Consequently, $\Sigma_r(C_n)$ is a vertex decomposable simplicial complex. 
		
		The implications \emph{(1) $\Rightarrow$ (2) $\Rightarrow$(3)} follows from the known hierarchy of conditions in \Cref{VdShCm}. Also, \emph{(3) $\Rightarrow$(4)} follows from \Cref{cycle lemma}, and \emph{(1) $\Rightarrow$ (5)} follows from \Cref{V.D implies chordality}. Moreover, \emph{(5) $\Rightarrow$ (3)} follows from \cite[Theorem 3.3]{SBCC} and \cite[Theorem 3]{EagonReiner1998}. This completes the proof of the theorem.
	\end{proof}

In the theorem below, we examine the complex $\Sigma_r(G)$ when $G = C_n^c$, the complement of the cycle graph $C_n$.
    
	\begin{theorem}\label{cycle complement}
		Let $n \geq 3$ and $r \geq 3$. Then $\Sigma_r(C_n^{c})$ is vertex decomposable.
	\end{theorem}
	
	\begin{proof}
		It is easy to see that for $n=3,4$,  
		$\Sigma_r(C_n^{c})$ is the void complex, and thus vertex decomposable. Therefore, we may assume $n\ge 5$. Let $x \in V(C_n^{c})$ be any vertex. Then the induced subgraph $C_n^{c} \setminus N_{C_n^{c}}[x]$ consists of a single circuit between the vertices that are adjacent to $x$ in $C_n$. In particular, for $r \geq 3$, we have
		\[
		E(\Con_r(C_n^{c} \setminus N_{C_n^{c}}[x])) = \emptyset.
		\]
		Thus, $x$ is a shedding vertex of $\Sigma_r(C_n^{c})$. For such an $x$, we have the following:
		
		\begin{itemize}
			\item $\del_{\Sigma_r(C_n^c)}(x) = \Sigma_{r-1}(\{x\},C_n^c)$, by \Cref{Eq. conditions for lk and del}, and this complex is vertex decomposable, by \Cref{Vertex decoposable}.
			
			\item $\lk_{\Sigma_r(C_n^c)}(x) = \Sigma_r(C_n^c \setminus x)$, where $C_n^{c} \setminus \{x\}$ is isomorphic to $P_{n-1}^{c}$, the complement of the path on $n-1$ vertices. Since $P_{n-1}^{c}$ is a co-chordal graph, it follows from \Cref{V.D of Sigma_r(G) if G is co-chordal} $\Sigma_r(C_n^{c} \setminus x)$ is vertex decomposable.
		\end{itemize}
		
		Hence, $\Sigma_r(C_n^c)$ is vertex decomposable.
	\end{proof}
	
	\begin{remark}
		In \cite[Theorem 5.4]{FVDH}, it was shown that the $r$-connected ideal $I_{\mathrm{Ind}_r(C_n^c)}$ has a linear resolution, equivalently that $\Sigma_r(G)$ is Cohen–Macaulay. In \Cref{cycle complement} we have strengthened this result by showing that the complex is, in fact, vertex decomposable.
	\end{remark}
	\begin{remark}
		It is well-known that for $r=2$ and $n\ge 4$, the complex $\Sigma_2(C_n^{c})$ is not even Cohen-Macaulay (follows from Fr\"oberg's theorem). Hence, the condition $r\ge 3$ is necessary in \Cref{cycle complement} in order to show the vertex decomposability of $\Sigma_r(C_n^{c})$.
	\end{remark}

Our next two results concern the ladder graph $P_n\times P_2$ and the graph $P_n\times P_3$. Recall that the cartesian product (or simply {\it product}) of two disjoint graphs $G_1$ and $G_2$, denoted by $G_1\times G_2$, consists of the vertex set $V(G_1\times G_2)=\{(x,y)\mid x\in V(G_1)\text{ and }y\in V(G_2)\}$, and the circuit set $E(G_1\times G_2)=\{\{(x,y),(a,b)\}\mid \text{either }x=a\text{ with }\{y,b\}\in E(G_2)\text{ or }y=b\text{ with }\{x,a\}\in E(G_1)\}$.
    
	For the ladder graph $P_n \times P_2$, the following characterization holds.
	
	\begin{theorem}\label{ladder graph}
		Let $G=P_n \times P_2$. For $r\ge 2$, the following are equivalent:
		\begin{enumerate}
			\item $\Sigma_r(G)$ is vertex decomposable.
			\item $\Sigma_r(G)$ is shellable.
			\item $\Sigma_r(G)$ is Cohen-Macaulay.
			\item $G$ is r-gap-free.
			\item $n \leq r$.
			\item The clutter $\Con_r(G)$ is co-chordal.    \end{enumerate}
	\end{theorem}
	
	\begin{proof}
		Let $V(P_n)=\{x_1,\ldots,x_n\}$ and $V(P_2)=\{y_1,y_2\}$.
		
		We first show \emph{(4) $\Rightarrow$ (5)}. Suppose if possible, \emph{(4)} holds but $n\ge r+1$. Our aim is to show that $G$ is not $r$-gap-free. Consider the following two cases:
		
		\noindent
		{\bf Case I:} $r$ is even, i.e., $r=2s$ for some $s\ge 1$. Let us consider the two subsets of vertices
		\begin{align*}
			\E_1&=\{(x_i,y_j)\mid i\in\{n-(s-1),\ldots,n-1,n\},j\in[2]\},\\
			\E_2&=\{(x_i,y_j)\mid i\in[s],j\in[2]\}.
		\end{align*}
		Observe that $\E_1,\E_2\in E(\Con_r(G))$ and they form a gap.
		
		\noindent
		{\bf Case II:} $r$ is odd, i.e., $r=2s+1$ for some $s\ge 1$. Let us consider the two subsets of vertices
		\begin{align*}
			\E_1&=\{(x_i,y_j),(x_{n-s},y_1)\mid i\in\{n-(s-1),\ldots,n-1,n\},j\in[2]\},\\
			\E_2&=\{(x_i,y_j),(x_{s+1},y_2)\mid i\in[s],j\in[2]\}.
		\end{align*}
		As before, one can see that $\E_1,\E_2\in E(\Con_r(G))$ and they form a gap. Thus, we have \emph{(4) $\Rightarrow$ (5)}.
		
		Next, we proceed to show \emph{(5) $\Rightarrow$ (1)}. To prove this, we use \Cref{V.D of Sigma_r(G)}.
		
		\noindent
		{\bf Case I:} $n$ is even, i.e., $n=2m$ for some $m\ge 1$. In this case, we take $a_1=(x_m,y_1), a_2=(x_m,y_2)$, and show that the sequence of vertex $a_1,a_2$ satisfy condition (1) in \Cref{V.D of Sigma_r(G)}. Indeed, observe that $G\setminus N_G[a_1]=H_1\sqcup H_2$, where $H_1$ is the induced subgraph of $G$ on the vertex set 
    \[
    \{(x_i,y_j),(x_{m+1},y_2)\mid i\in [n]\setminus [m+1],j\in[2]\},
    \]
    and $H_2$ is the induced subgraph of $G$ on the vertex set 
    \[
    \{(x_i,y_j),(x_{m-1},y_2)\mid i\in [m-2],j\in[2]\}.
    \]
    Thus, $E(\Con_r(G\setminus N_G[a_1]))=\emptyset$. Next, consider the graph $G_1=G\setminus a_1$. Then analyzing the graph $G_1$ as above one can see that $E(\Con_r(G_1\setminus N_{G_1}[a_2]))=\emptyset$, and
		\begin{align*}
			|E(\Con_r(G\setminus \{a_1,a_2\}))|=\begin{cases}
				1&\text{ if }r=2m\\
				0&\text{ if }r>2m.
			\end{cases}
		\end{align*}
		Thus, in any case, $|E(\Con_r(G\setminus \{a_1,a_2\}))|\le 1$, as desired. 
		
		\noindent
		{\bf Case II:} $n$ is odd, i.e., $n=2m+1$ for some $m\ge 1$. In this case, we take $a_1=(x_{m+1},y_1), a_2=(x_{m+1},y_2)$ and show that the sequence of vertex $a_1,a_2$ satisfy condition (1) in \Cref{V.D of Sigma_r(G)}. Indeed, observe that $G\setminus N_G[a_1]=H_1\sqcup H_2$, where $H_1$ is the induced subgraph of $G$ on the vertex set 
        \[
        \{(x_i,y_j),(x_{m+2},y_2)\mid i\in [n]\setminus [m+2],j\in[2]\},
        \]
        and $H_2$ is the induced subgraph of $G$ on the vertex set 
        \[
        \{(x_i,y_j),(x_{m},y_2)\mid i\in [m-1],j\in[2]\}.
        \]
        Thus, $E(\Con_r(G\setminus N_G[a_1]))=\emptyset$. Next, consider the graph $G_1=G\setminus a_1$. Observe that $E(\Con_r(G\setminus \{a_1,a_2\}))=\emptyset$, and also $E(\Con_r(G\setminus N_G[a_1]))=\emptyset$. Thus, we have \emph{(5) $\Rightarrow$ (1)}.
		
		The implications \emph{(1) $\Rightarrow$ (2) $\Rightarrow$(3)} follows from the known hierarchy of conditions in \Cref{VdShCm}. Also, \emph{(3) $\Rightarrow$(4)} follows from \cite[Theorem 3]{EagonReiner1998} and \cite[Theorem 1.4]{HaWoodroofe2014}. \emph{(1) $\Rightarrow$ (6)} follows from \Cref{V.D implies chordality}. Moreover, \emph{(6) $\Rightarrow$ (3)} follows from \cite[Theorem 3.3]{SBCC} and \cite[Theorem 3]{EagonReiner1998}. This completes the proof of the theorem.
	\end{proof}

   A corresponding result also holds for the grid graph $P_n \times P_3$.

    \begin{theorem}\label{3ladder graph}
		Let $G = P_n \times P_3$. For $r \geq 2$ and $n \geq 3$, the following are equivalent:
		\begin{enumerate}
			\item $\Sigma_r(G)$ is vertex decomposable.
			\item $\Sigma_r(G)$ is shellable.
			\item $\Sigma_r(G)$ is Cohen-Macaulay.
			\item $2n \le r$.
			\item The clutter $Con_r(G)$ is co-chordal.
		\end{enumerate}
	\end{theorem}
	\begin{proof}
		Let $V(P_n) = \{x_1,x_2,\ldots,x_n\}$ and $V(P_3) = \{y_1,y_2,y_3\}$. We first show that \emph{(3)} $\Rightarrow$ \emph{(4)}. Suppose, for a contradiction, that \emph{(3)} holds but $2n > r$. Consider the vertex set
		\[
		A = \{(x_i,y_1), (x_i,y_3) : 1 \leq i \leq n \} \cup \{(x_1,y_2), (x_n,y_2)\}.
		\]
		Then $G[A]$ is an induced cycle of length $2n+2$. By \cite[Theorem~5.3]{LQCG}, $\Sigma_r(G)$ cannot be Cohen–Macaulay since $r+2 < 2n+2$. This contradiction shows that \emph{(3)} $\Rightarrow$ \emph{(4)}.
		
		Now we show \emph{(4)} $\Rightarrow$ \emph{(1)}. Assume $r \geq 2n$. To prove that $\Sigma_r(G)$ is vertex decomposable, we use \Cref{V.D of Sigma_r(G)}. There are two cases to consider:
		
		\medskip	
		\noindent
		\textbf{Case I:} $n$ is even, i.e., $n = 2m$ for some $m \geq 2$. In this case, we will show that 
        \begin{multline*}
a_1 = (x_m, y_2),\; a_2 = (x_{m+1}, y_2),\; a_3 = (x_{m-1}, y_2),\; 
a_4 = (x_{m+2}, y_2),\; \ldots,\\
a_{n-3} = (x_2, y_2),\; a_{n-2} = (x_{n-1}, y_2),\;
a_{n-1} = (x_1, y_2),\; a_n = (x_n, y_2)
\end{multline*}
        is a required sequence of vertices which will allow us to conclude that $\Sigma_r(G)$ is vertex decomposable (by \Cref{V.D of Sigma_r(G)}). Fix an index $j$ with $1 \leq j \leq n$. Set 
		\[G_{j-1} : = G \setminus \{a_1,\ldots,a_{j-1}\} \; \quad \; H_{j-1} : = G_{j-1} \setminus N_{G_{j-1}}[a_j] \; \quad \; \text{with}\quad \; G_0 = G.\]
		Observe that it suffices to show that $E(\Con_r(H_{j-1})) = \emptyset$ only for odd $j \in [n]$. By symmetry, we can guarantee that $E(\Con_r(H_{j-1})) = \emptyset$ holds for even $j \in [n]$. 
		
		Now, $H_0 = L_0 \sqcup R_0$, where $L_0$ is the induced subgraph of $G$ on the vertex set 
        \[
        \{(x_i,y_1), (x_i,y_3) : i \in [m-1]\} \cup \{(x_i,y_2) : i \in [m-2]\},
        \]
        of cardinality $3(m-1) -1$, and $R_0$ is the induced subgraph of $G$ on the vertex set 
        \[
        \{(x_i,y_1), (x_i,y_3) : i \in [n] \setminus [m]\} \cup \{(x_i,y_2) : i \in [n] \setminus [m+1]\},
        \]
        of cardinality $3m-1$. Thus, $E(\Con_r(H_{0})) = \emptyset$. Next, consider the graph $G_2 = G \setminus \{(x_m,y_2),(x_{m+1},y_2)\}$. Then, $H_2 = L_2 \sqcup R_2$, where $L_2$ is an induced subgraph of $L_0$, and $R_2$ is the induced subgraph of $G$ on the vertex set 
        \[
        \{(x_i,y_1), (x_i,y_3) : i \in [n] \setminus [m-1]\} \cup \{(x_i,y_2) : i \in [n] \setminus [m+1]\},
        \]
        of cardinality $3m+1$. Thus, $E(\Con_r(H_{2})) = \emptyset$. For odd $j \in [n] \setminus \{1,3\}$, we have
		
		\[H_{j-1} = L_{j-1} \sqcup R_{j-1},\]
		where $L_{j-1}$ is the induced subgraph of $L_{j-3}$ and $R_{j-1}$ is the induced subgraph of $G$  on the vertex set 
        \[
        \left\{(x_i,y_1), (x_i,y_3) : i \in [n] \setminus \left[m-\frac{j-1}{2}\right]\right\} \cup \left\{(x_i,y_2) : i \in [n] \setminus \left[m+\frac{j-1}{2}\right]\right\}
        \]
        of cardinality $3m+1+\frac{j-3}{2}$. Observe that, $3m+1+\frac{j-3}{2} \le 4m -1 < 4m = 2n$. Therefore, $E(\Con_r(G_{j-1})) = \emptyset$. After eliminating $a_i$ for every $i \in [n]$, we obtain a graph $G'=G \setminus \{(x_i,y_2): i\in[n]\}$ such that $E(\Con_r(G'))=\emptyset$. Hence, by \Cref{V.D of Sigma_r(G)}, $\Sigma_r(G)$ is vertex decomposable.
		
		\medskip
		\noindent 
		\textbf{Case II:} $n$ is odd, i.e., $n=2m+1$ for some $m \geq 1$.  In this case, we will show that 
        \begin{multline*}
a_1 = (x_{m+1}, y_2),\; a_2 = (x_{m}, y_2),\; a_3 = (x_{m+2}, y_2),\; 
a_4 = (x_{m-1}, y_2),\; a_5 = (x_{m+3}, y_2),\; \ldots,\\
a_{n-1} = (x_1, y_2),\; a_{n} = (x_n, y_2)
\end{multline*}
is a required sequence of vertices. As before, fix an integer $j$ with $1 \leq j \leq n$. Set 
		\[G_{j-1} : = G \setminus \{a_1,\ldots,a_{j-1}\} \; \quad \; H_{j-1} : = G_{j-1} \setminus N_{G_{j-1}}[a_j] \; \quad \; \text{with} \; G_0 = G.\]
		Observe that it suffices to show that $E(\Con_r(H_{j-1})) = \emptyset$ only for $j=1$ and even $j \in [n]$. By symmetry, we can guarantee that $E(\Con_r(H_{j-1})) = \emptyset$ holds for odd $j \in [n] \setminus \{1\}$. 
		
		Now, $H_0 = L_0 \sqcup R_0$, where $L_0$ is the induced subgraph of $G$ on the vertex set 
        \[
        \{(x_i,y_1), (x_i,y_3) : i \in [m]\} \cup \{(x_i,y_2) : i \in [m-1]\},
        \]
        of cardinality $3m -1$, and $R_0$ is the induced subgraph of $G$ on the vertex set 
        \[
        \{(x_i,y_1), (x_i,y_3) : i \in [n] \setminus [m+1]\} \cup \{(x_i,y_2) : i \in [n] \setminus [m+2]\},
        \]
        of cardinality $3m-1$. Thus, $E(\Con_r(H_{0})) = \emptyset$.  Next, consider the graph $G_1 = G \setminus \{(x_{m+1},y_2)\}$. Then, $H_1 = L_1 \sqcup R_1$, where $L_1$ is the induced subgraph of $G$ which is the subgraph of $L_0$, and $R_1$ is the induced subgraph of $G$ on the vertex set \[
        \{(x_i,y_1), (x_i,y_0) : i \in [n] \setminus [m]\} \cup \{(x_i,y_2) : i \in [n] \setminus [m+1]\},
        \]
        of cardinality $3m+2$. Thus, $E(\Con_r(H_{1})) = \emptyset$. For even $j \in [n-2] \setminus \{1,2\}$, we have $H_{j-1} = L_{j-1} \sqcup R_{j-1}$,
		where $L_{j-1}$ is the induced subgraph of $L_{j-3}$ and $R_{j-1}$ is the induced subgraph of $G$  on the vertex set 
        \[
        \left\{(x_i,y_1), (x_i,y_3) : i \in [n] \setminus \left[m-\frac{j-2}{2}\right]\right\} \cup \left\{(x_i,y_2) : i \in [n] \setminus \left[m+\frac{j}{2}\right]\right\},
        \]
        of cardinality $3m+2+\frac{j-2}{2}$. Observe that, $3m+2+\frac{j-2}{2} \le 4m < 4m+2 = 2n$. Therefore, $E(\Con_r(G_{j-1})) = \emptyset$. After eliminating $a_i$ for every $i \in [n-2]$, we we obtain a graph $G'=G \setminus \{(x_i,y_2): i\in[n]\}$ such that $E(\Con_r(G'))=\emptyset$. Hence, by \Cref{V.D of Sigma_r(G)}, $\Sigma_r(G)$ is vertex decomposable, as required.

        The implications \emph{(1) $\Rightarrow$ (2) $\Rightarrow$(3)} follows from the known hierarchy of conditions in \Cref{VdShCm}. \emph{(1) $\Rightarrow$ (5)} follows from \Cref{V.D implies chordality}, and \emph{(5) $\Rightarrow$ (3)} follows from \cite[Theorem 3.3]{SBCC} and \cite[Theorem 3]{EagonReiner1998}. This completes the proof of the theorem.
	\end{proof}

For the sake of completeness we state some obvious Corollaries and a Remark about a class examples that could potentially separate shellabality and vertex decomposability. 

\begin{corollary}\label{cor:field-independent}
Let $G$ be a graph belonging to any of the classes listed in
\Cref{V.D of Sigma_r(G) if G is chordal}, \Cref{V.D of Sigma_r(G) if G is co-chordal}, \Cref{V.D of Sigma_r(G) if G is cograph}, \Cref{cycle vd theorem}, \Cref{cycle complement}, \Cref{ladder graph}, and \Cref{3ladder graph}, and let $r \ge 2$.
Then the following statements are equivalent and independent of the
characteristic of the base field $K$:
\begin{enumerate}
  \item the $r$-connected ideal $I_{\mathrm{Ind}_r(G)}$ has a linear resolution over $K$ (and has regularity $r+1$);
  \item the simplicial complex $\Sigma_r(G)$ is Cohen--Macaulay over $K$;
  \item the clutter $\operatorname{Con}_r(G)$ is co-chordal.
\end{enumerate}
\end{corollary}

\begin{corollary}\label{cor:seq-cm}
Let $G$ be a graph in any of the classes stated above in \Cref{cor:field-independent} and let $r \ge 2$.
If $\Sigma_r(G)$ is vertex decomposable, then $\Sigma_r(G)$ is
sequentially Cohen--Macaulay.
\end{corollary}

\begin{corollary}\label{cor:gamma-bound}
Let $G$ be a graph in any of the classes stated above in \Cref{cor:field-independent} and let $r \ge 2$.
If $\Sigma_r(G)$ is Cohen--Macaulay, then the induced matching number of
the $r$-connected clutter satisfies $\gamma_r(G) \le 1$.
\end{corollary}

\begin{remark}\label{general grid counter}
In view of \Cref{ladder graph} and \Cref{3ladder graph}, one might conjecture that for a general grid graph $P_n \times P_m$, the properties of vertex decomposability, shellability, and Cohen–Macaulayness of the complex $\Sigma_r(P_n \times P_m)$ are equivalent. 
However, this is not the case. Computational evidence using SageMath shows that $\Sigma_{10}(P_4 \times P_4)$ is shellable but not vertex decomposable.   
\end{remark}

We end the section by a short discussion on algorithmic aspects of deciding whether a given $\Sigma_r(G)$ is vertex decomposable. 
Note that the algorithm is primarily of theoretical interest and serves as a certificate-based decision procedure rather than a practical implementation for large graphs. 

\begin{corollary}\label{algorithm}
    Let $G$ be a graph on $n$ vertices and let $r\geq 2$ be a fixed positive integer. 
    There exists an algorithm that determines whether $\Sigma_r(G)$ is vertex decomposable in polynomial time for all values of $r$. 
\end{corollary}

\begin{proof}
    By \Cref{V.D of Sigma_r(G)}, a vertex $v$ of a graph $G$ is a shedding vertex for
$\Sigma_r(G)$ if and only if the graph
\[
H_v := G \setminus N[v]
\]
contains no connected induced subgraph on $r$ vertices.  Equivalently,
$v$ is shedding if and only if every connected component of $H_v$ has
strictly fewer than $r$ vertices.

Fix a vertex $v \in V(G)$.  The sizes of the connected components of
$H_v$ can be determined by a standard traversal of the graph $H_v$,
during which each vertex and each edge (circuit) is visited at most once.
Consequently, whether $v$ satisfies the shedding condition can be
decided in time proportional to $|V(G)|+|E(G)|$.

To decide whether $\Sigma_r(G)$ is vertex decomposable, one proceeds
iteratively as follows.  Given the current graph $G$, one tests each
vertex $v \in V(G)$ to determine whether it is shedding.  If no such
vertex exists, then by \Cref{V.D of Sigma_r(G)} the complex $\Sigma_r(G)$ is not
vertex decomposable.  If a shedding vertex $v$ is found, then $v$ is
removed from the graph, and the same procedure is applied to the
resulting induced subgraph.

By \Cref{V.D of Sigma_r(G)}, this greedy elimination procedure is correct whenever
$\Sigma_r(G)$ is vertex decomposable.  Since at most $n$ vertices are
removed, and at each stage at most $n$ vertices are tested, each test
requiring time proportional to $|V(G)|+|E(G)|$, the overall procedure
terminates after a number of steps polynomial in the size of $G$.
\end{proof}

\begin{remark}
Although the definition of a shedding vertex is phrased in terms of connected induced subgraphs on $r$ vertices, it is not necessary to enumerate such subgraphs explicitly.  
Indeed, a graph contains a connected induced subgraph on $r$ vertices if and only if it has a connected component of size at least $r$.  
Thus, the verification of the shedding condition reduces to an elementary connectivity check, rather than a combinatorial search among $r$-vertex subsets.

In particular, the procedure yields a decision method for vertex
decomposability of $\Sigma_r(G)$ whose running time is polynomial in the
number of vertices and edges (circuits) of $G$, for arbitrary $r$.
\end{remark}

\begin{remark}\label{non-vd}
    The above Corollary also provides a certificate of non-vertex-decomposability: if the algorithm returns FALSE, then no shedding sequence exists, and the complex is provably not vertex decomposable. This can be used to systematically search for boundary cases or counterexamples to \Cref{conj1}.  
\end{remark}

    \section{Concluding remarks}\label{con-rem}

	For various classes of graphs, we have observed in \Cref{Section 4 Froberg} that $\Sigma_{r}(G)$ is vertex decomposable if and only if $\Con_r(G)$ is co-chordal. Below we provide one example of a graph $G$, which happens to be a gap-free graph, such that $\Con_3(G)$ is co-chordal but  $\Sigma_{3}(G)$ is not vertex decomposable.
	
	\begin{figure}[H]
		\centering
		\begin{tikzpicture}
			[scale=0.35, vertices/.style={draw, fill=black, circle, inner sep=1.5pt}]
			\node[vertices, label=above:{$x_1$}] ($x_1$) at (0,8)  {};
			\node[vertices, label=left:{$x_2$}] ($x_2$) at (-5,4.8)  {};
			\node[vertices, label=right:{$x_3$}] ($x_3$) at (-2,4.2)  {};
			\node[vertices, label=right:{$x_4$}] ($x_4$) at (1.5,4.4)  {};
			\node[vertices, label=right:{$x_5$}] ($x_5$) at (5,4.6)  {};
			\node[vertices, label=below:{$x_6$}] ($x_6$) at (-4.1,0.4)  {};
			\node[vertices, label=below:{$x_7$}] ($x_7$) at (-1,0.4)  {};
			\node[vertices, label=below:{$x_8$}] ($x_8$) at (2.4,0.4)  {};
			
			\foreach \to/\from in {$x_1$/$x_2$,$x_1$/$x_3$,$x_1$/$x_4$,$x_1$/$x_5$,$x_2$/$x_6$,$x_2$/$x_8$,$x_3$/$x_6$,$x_3$/$x_8$,$x_4$/$x_7$,$x_5$/$x_7$,$x_6$/$x_7$,$x_7$/$x_8$}
			\draw [-] (\to)--(\from);
		\end{tikzpicture}\caption{}\label{fig:gap free graph which is not claw free}
	\end{figure}
\begin{proposition}\label{gap-free example}
    Let $G$ be the graph in \Cref{fig:gap free graph which is not claw free}. Then $G$ is gap-free, $\Con_3(G)$ is co-chordal, and $\Sigma_3(G)$ is shellable but not vertex decomposable.
\end{proposition}
\begin{proof}
    It is easy to see that $G$ is a gap-free graph. On the other hand, by \Cref{Eq. Condition for a shedding vertex}, $\Sigma_{3}(G)$ does not contain any shedding vertex, and thus, $\Sigma_{3}(G)$ is not vertex decomposable. However, $\mathcal C=(\Con_3(G))^c$ is chordal since $e_1,e_2,\ldots,e_{20}$ is a sequence of maximal subcircuits of $\mathcal C$ such that $e_1$ is a simplicial maximal subcircuit of $\mathcal C$, $e_i$ is a simplicial maximal subcircuit of $(((\mathcal C\setminus_d e_1)\setminus_d e_2)\setminus_d\cdots)\setminus_d e_i$ for all $i>1$, and $(((\mathcal C\setminus_d e_1)\setminus_d e_2)\setminus_d\cdots)\setminus_d e_{20}=\emptyset$, where
	\[
	\begin{array}{rrrrrrrrrr}
		& e_1 = \{x_1,x_2\} & e_2= \{x_1,x_3\} & e_3= \{x_1,x_4\} & e_4= \{x_1,x_5\} & e_5= \{x_2,x_6\} \\ & e_6= \{x_2,x_8\}		
		& e_7 = \{x_3,x_6\} & e_8= \{x_3,x_8\} & e_9= \{x_4,x_6\}& e_{10}= \{x_4,x_7\} \\
		& e_{11}= \{x_4,x_8\} & e_{12}= \{x_5,x_6\}
		 & e_{13} = \{x_5,x_7\} & e_{14}= \{x_6,x_7\} & e_{15}= \{x_6,x_8\}\\ & e_{16}= \{x_1,x_7\} & e_{17}= \{x_2,x_7\} &e_{18}= \{x_2,x_3\}
		  & e_{19} = \{x_2,x_4\} & e_{20}= \{x_3,x_4\}.
	\end{array}
	\]
	Note that the complex $\Sigma_{3}(G)$ is shellable, as the following ordering of its facets is a shelling order:
	
	\[
	\begin{aligned}
		&\{x_4,x_5,x_6,x_7,x_8,\},\{x_3,x_5,x_6,x_7,x_8\},\{x_2,x_5,x_6,x_7,x_8\},\{x_3,x_4,x_6,x_7,x_8\},\{x_2,x_4,x_6,x_7,x_8\},\\
		&\{x_2,x_3,x_6,x_7,x_8\},\{x_3,x_4,x_5,x_7,x_8\},\{x_2,x_4,x_5,x_7,x_8\},\{x_1,x_4,x_5,x_7,x_8\},\{x_2,x_3,x_5,x_6,x_8\},\\
		&\{x_2,x_3,x_4,x_6,x_8\},\{x_1,x_3,x_4,x_5,x_8\},\{x_1,x_2,x_4,x_5,x_8\},\{x_1,x_2,x_3,x_5,x_8\},\{x_1,x_2,x_3,x_6,x_8\},\\
		&\{x_1,x_2,x_3,x_4,x_8\},\{x_3,x_4,x_5,x_6,x_7\},\{x_2,x_4,x_5,x_6,x_7\},\{x_1,x_3,x_4,x_5,x_7\},\{x_1,x_2,x_4,x_5,x_7\},\\
		&
		\{x_1,x_2,x_3,x_5,x_6\},\{x_1,x_3,x_4,x_5,x_6\},\{x_1,x_4,x_5,x_6,x_7\},\{x_1,x_2,x_4,x_5,x_6\},\{x_1,x_2,x_3,x_4,x_6\},\\
		&\{x_1,x_2,x_3,x_4,x_5\}.
	\end{aligned}
	\]
    \end{proof}
    \begin{remark}
       By \Cref{gap-free example}, we have an example of a gap-free graph $G$ such that $\Sigma_{3}(G)$ is shellable but not vertex decomposable. This shows that the answer to the following question, analogous to \cite[Question 7.5]{LQCG}, is `no': if $G$ is gap-free, then is $I_{\mathrm{Con}_r(G)}$ vertex splittable for all $r\ge 3$? 
    \end{remark}

    Our next result provides an example of an $r$-gap-free unicyclic graph $G$ for each $r\ge 3$ such that the complex $\Sigma_r(G)$ is not vertex decomposable, but it is shellable. Indeed, let $G$ be the graph with the vertex set and circuit set defined as follows: 
\begin{equation}\label{eq3146}
	\begin{split}
		V(G) &= \{x_1, x_2, \ldots, x_r, x_{r+1}, x_{r+2}\} \cup T, \quad \text{with} \quad T = \left\{ x_{r+3}, x_{r+4}, \ldots,  x_{(r+2)+\left\lfloor \tfrac{r+3}{2} \right\rfloor} \right\}, \\
		E(G) &= \left\{\{x_{r+2}, x_1\}, \{x_i, x_{i+1}\} \mid i \in [r+1] \right\}\cup \left\{ \left\{x_j, x_{r+2 + \tfrac{j+1}{2}} \right\} \mid j \in [r+2], j\text{ odd integer} \right\}.
		\end{split}
\end{equation}

The graph $G$ for $r=3,4$ is depicted in \Cref{fig:unicyclic-r3-r4}.
	\begin{figure}[H]
	\centering
	
	\begin{minipage}[t]{0.48\textwidth}
		\centering
		\begin{tikzpicture}
			[scale=0.4, vertices/.style={draw, fill=black, circle, inner sep=1.5pt}]
			
			\node[vertices, label=left:{$x_1$}] (1) at (0,4) {};
			\node[vertices, label=right:{$x_2$}] (2) at (2.5,2.5) {};
			\node[vertices, label=right:{$x_3$}] (3) at (1.5,0) {};
			\node[vertices, label=below left:{$x_4$}] (4) at (-1.5,0) {};
			\node[vertices, label=above:{$x_5$}] (5) at (-2.5,2.5) {};
			
			\node[vertices, label=above:{$x_6$}] (6) at (0,5.5) {};
			\node[vertices, label=below:{$x_7$}] (7) at (2.4,-1.2) {};
			\node[vertices, label=left:{$x_8$}]  (8) at (-4,2.5) {};
			
			\draw (1) -- (2) -- (3) -- (4) -- (5) -- (1);
			\draw (1) -- (6);
			\draw (3) -- (7);
			\draw (5) -- (8);
			
		\end{tikzpicture}
		\caption*{(a) $r = 3$}
	\end{minipage}
	\hfill
	\begin{minipage}[t]{0.48\textwidth}
		\centering
		\begin{tikzpicture}
			[scale=0.4, vertices/.style={draw, fill=black, circle, inner sep=1.5pt}]
			
			\node[vertices, label=left:{$x_1$}]       (1) at (0,4.5) {};
			\node[vertices, label=above right:{$x_2$}] (2) at (2,3) {};
			\node[vertices, label=below:{$x_3$}]       (3) at (2,0.5) {};
			\node[vertices, label=below:{$x_4$}]       (4) at (0,-1) {};
			\node[vertices, label=below:{$x_5$}]        (5) at (-2,0.5) {};
			\node[vertices, label=above left:{$x_6$}]  (6) at (-2,3) {};
			
			\node[vertices, label=above:{$x_7$}]  (7) at (0,6) {};
			\node[vertices, label=right:{$x_8$}]  (8) at (4,0) {};
			\node[vertices, label=left:{$x_9$}]   (9) at (-4,0) {};
			
			\draw (1) -- (2) -- (3) -- (4) -- (5) -- (6) -- (1);
			\draw (1) -- (7);
			\draw (3) -- (8);
			\draw (5) -- (9);
		\end{tikzpicture}
		\caption*{(b) $r = 4$}
	\end{minipage}
	
	\caption{Leaves attached to odd-numbered cycle vertices for (a) $r = 3$ and (b) $r = 4$.}
	\label{fig:unicyclic-r3-r4}
\end{figure}
	
	\begin{proposition}\label{shell but not VD}
		Let $G$ denote the graph with the vertex and circuit set described in \Cref{eq3146} and $r \ge 3$. Then $G$ is an $r$-gap-free unicyclic graph such that the complex $\Sigma_{r}(G)$ is shellable but not vertex decomposable.
	\end{proposition}
	\begin{proof}
For each $r\ge 8$, one can observe that $(r+2)+\lfloor\frac{r+3}{2}\rfloor<2r$. Thus, analyzing the remaining $3\le r\le 7$ cases individually, we see that $G$ is $r$-gap-free. 
	
	Our aim now is to show that $\Sigma_r(G)$ is not vertex decomposable. Indeed, for any $v\in V(G)$, we have $|N_G[v]|\le 3$. Thus, $E(\Con_r(G\setminus N_G[v])) \neq \emptyset$, and hence, by \Cref{Eq. Condition for a shedding vertex}, $\Sigma_r(G)$ has no shedding vertex. Consequently, $\Sigma_r(G)$ is not vertex decomposable.
    
		To show that $\Sigma_{r}(G)$ is shellable, we need to find a shelling order on the set of facets of $\Sigma_{r}(G)$, and we do this by defining first a special total order on $\mathcal P(T)$, the power set of $T$. Note that the elements $\{i\mid x_i\in T\}$ have natural order induced from the set of natural numbers, which we denote by $<_{\mathbb N}$. Now, we endow $\mathcal P(T)$ with a special order $\lessdot$ which is described in the following steps. 
		 
		 Let $I,J\in \mathcal P(T)$, where $I=\{x_{i_1}\ldots,x_{i_n}\}$ and $J=\{x_{j_1},\ldots,x_{j_m}\}$. Let $i=\max_{<_{\mathbb N}}\{i_1,\ldots,i_n\}$ and $j=\max_{<_{\mathbb N}}\{j_1,\ldots,j_m\}$ .
		 
		 \begin{enumerate}
		 	\item[{\bf Step 0:}] Declare $\emptyset\lessdot I$ for any $\emptyset\neq I\in \mathcal P(T)$.
		 	
		 	\item[{\bf Step 1:}] If $i<_{\mathbb N}j$, then declare $I\lessdot J$.
		 	
		 	\item[{\bf Step 2:}] If $i=j$, and $|I|<_{\mathbb N} |J|$, then declare $I\lessdot J$.
		 	
		 	\item[{\bf Step 3:}] If $i=j$ and $|I|=|J|$, then consider $I'=I\setminus\{x_i\}$ and $J'=J\setminus\{x_j\}$, and go back to Step 1. If $I'\lessdot J'$, then declare $I\lessdot J$.
		 \end{enumerate}
		As an example, if we consider the set $T=\{x_1,x_2,x_3,x_4\}$, then the elements of $\mathcal P(T)$ are ordered as follows:
		 \begin{align*}
		 \emptyset\lessdot\{x_1\}&\lessdot\{x_2\}\lessdot\{x_1,x_2\}\lessdot\{x_3\}\lessdot\{x_1,x_3\}\lessdot\{x_2,x_3\}\lessdot\{x_1,x_2,x_3\}\lessdot\{x_4\}\lessdot\{x_1,x_4\}\lessdot\\
		 &\{x_2,x_4\}\lessdot\{x_3,x_4\}\lessdot\{x_1,x_2,x_4\}\lessdot\{x_1,x_3,x_4\}\lessdot\{x_2,x_3,x_4\}\lessdot\{x_1,x_2,x_3,x_4\}.
		 \end{align*}
		 It is easy to see that $\lessdot$ is a total order on the elements of $\mathcal P(T)$. We now proceed to provide a shelling order $<$ on $\mathcal F(\Sigma_{r}(G))$ based on the ordering $(\mathcal P(T),\lessdot)$.
		 
		 Let $F\in \mathcal F(\Sigma_{r}(G))$. Then $F^c\in E(\Con_r(G))$. For each $I\in\mathcal P(T)$, we define $\mathcal F_I=\{F\in \mathcal{F}(\Sigma_{r}(G))\mid F^c\cap T=I\}$. Observe that $\mathcal{F}(\Sigma_{r}(G))=\sqcup_{I\in\mathcal P(T)}\mathcal F_I$, and $\mathcal F_{\emptyset}=\mathcal F(\Sigma_{r}(C_{r+2}))$, where $C_{r+2}$ is the induced subgraph of $G$ on the vertex set $\{x_1,\ldots,x_{r+2}\}$. By \Cref{cycle vd theorem}, the facets of $\Sigma_{r}(C_{r+2})$ has a shelling order. Based on these observations, we define the following:   
		 \begin{enumerate}
		 	\item[$\bullet$] The facets in $\mathcal F_{\emptyset}$ are given a shelling order.
		 	
		 	\item[$\bullet$] For each fixed $\emptyset\neq I\in\mathcal P(T) $, the facets in $\F_I$ are given any arbitrarily chosen but fixed total order.
		 	
		 	\item[$\bullet$] For any two distinct $I,J\in \mathcal P(T)$ with $F_1\in \F_I$ and $F_2\in \F_J$, we have $F_1<F_2$ if and only if $I\lessdot J$.
		 \end{enumerate}
		 
		 \noindent
		 {\bf Claim}: $(\mathcal F(\Sigma_{r}(G)), <)$ is a shelling order.
		 
		 \vspace{.1cm}
		 \noindent
		 {\it Proof of Claim}: By construction, if $F\in \F_{\emptyset}$ and $F$ is not the first facet in the given ordering, then for any $F'\in \F_{\emptyset}$ with $F'<F$, there exists some $F''\in \F_{\emptyset}$ with $F''<F$ such that $F'\cap F\subseteq F''\cap F$ and $\dim(F''\cap F)=\dim F-1$.
		 
		 Now, let $F\in \F_I$ for some $\emptyset\neq I\in\mathcal P(T)$. Take any $F_1\in \F(\Sigma_{r}(G))$ with $F_1<F$. Let $P:=F\cap C_{r+2}$ be the path with two end points $u,v\in V(C_{r+2})$. If $u\notin F_1$, then consider $F_u=(F\cup\{x\})\setminus\{u\}$, where $x\in I$. Observe that
		  \[F_u^c=(F^c\setminus\{x\})\cup\{u\}\in E(\Con_r(G)),\]
		   since $F^c\in E(\Con_r(G))$. Thus, $F_u\in \F(\Sigma_{r}(G))$ with $F_u<F$, $\dim(F_u\cap F)=\dim(F)-1$, and $F_1\cap F\subseteq F_u\cap F$. A similar argument can be given in case $v\notin F_1$.
		   
		   Next, we consider the case when $u,v\in F_1$. Then $F_1\cap F\supseteq P$ since $|V(G)|<2r+2$ for each $r\ge 4$, and the $r=3$ case can be easily checked separately. Let $Q$ be the path on the vertex set $V(C_{r+2}\setminus P)$, and define
		   \[
		   K=\{x_i\in T\mid \{x_i,z\}\in E(G)\text{ for some }z\in Q\}.
		   \]
		   Since $u,v\in F_1$ and $F_1^c\in E(\Con_r(G))$, we have $F_1\in \F_J$ for some $\emptyset\neq J\in \mathcal P(T)$ with $J\lessdot I$ and $J\neq I$. Moreover, $I,J\subseteq K$ and if $t=\max_{<_{\mathbb N}}\{i\mid x_i\in I\}$, then $J\subseteq \{x_i\in T\mid i\le t\}$. Observe that if $J\subseteq I$, then by our choice of ordering, $J=I$, since $u,v\in F_1$, which is a contradiction. Thus, $J\setminus I\neq\emptyset$, and hence, there exists $x_p\in J\setminus I$ such that $p<t$. In this case, if we take $\widehat F=(F\cup\{x_t\})\setminus\{x_p\}$, then 
		   \[
		   \widehat{F}^c=(F^c\setminus x_t)\cup\{x_p\}\in E(\Con_r(G)).
		   \]
		   Furthermore, $\widehat F\in \F(\Sigma_{r}(G))$ with $\widehat F<F$, $\dim(\widehat F\cap F)=\dim(F)-1$, and $F_1\cap F\subseteq \widehat F\cap F$. This completes the proof of the claim and subsequently, the proposition.
	\end{proof}

\section*{Open Problems and Future Directions}
We end the article by briefly outlining some open problems, in addition to the \Cref{conj1} stated in the \Cref{introduction}. 
\begin{itemize} 

  \item \textbf{Uniform, case-free arguments}. How to expand the ideas proving the equivalence statements for cographs, cycles, complements of cycles, and the grid/ladder graphs? For example, is there a uniform structural reason that co-chordality of $\mathrm{Con}_r(G)$ matches with the vertex decomposability of $\Sigma_r(G)$, in these families?

  \item \textbf{Cohen-Macaulay but not shellable $\Sigma_r(G)$.} Find $G,r$ with $\mathrm{Con}_r(G)$ co-chordal (hence linear resolution of the $r$-connected ideal) while the complex $\Sigma_r(G)$ fails to be shellable. Potential search spaces include circulant/Cayley graphs, prisms $C_n\times P_2$, toroidal grids $C_m\times C_n$, etc.

  \item \textbf{Beyond covered families.} Test the full equivalence (vertex decomposable $\Leftrightarrow$ shellable $\Leftrightarrow$ Cohen-Macaulay $\Leftrightarrow$ co-chordal $\mathrm{Con}_r(G)$) on chordal bipartite graphs (bipartite graph with every cycle of length at least $6$ has a chord) and comparability/permutation graphs, where simplicial structures and elimination orders might extend.
  Recently, Ghosh and Selvaraja \cite{ghoshSelva}, among other things, show that for co-chordal-cactus graphs, $(2K_2, C_4)$-free graphs and co-grid graphs, the associated $r$-connected ideal has minimal free resolution, for $r\geq 2$. They achieve this by proving that the corresponding $\Con_r(G)$ is chordal. 
  Hence, it would be interesting to see whether the $r$-co-connected complex is vertex decomposable?


  \item \textbf{Higher-independence topology.} Leverage Leray-number and homotopy-type results for higher independence complexes to sharpen our understanding (e.g., in case of circulant graphs, grids) and to identify obstructions to the Cohen-Macaulay property in $\Sigma_r(G)$ via Hochster’s formula.


\end{itemize}

\noindent
{\bf Acknowledgements.} 
The authors are partially supported by a grant from the Infosys Foundation.
The authors sincerely thank Margaret Bayer and Sheila Sundaram  for providing helpful comments on an earlier version of this manuscript. 
The authors are grateful to Geevarghese Philip for explaining the complexity aspect and helping them with the proof of \Cref{algorithm}. 

\subsection*{Data availability statement} Data sharing does not apply to this article as no new data were created or analyzed in this study.

 \subsection*{Conflict of interest} The authors declare that they have no known competing financial interests or personal relationships that could have appeared to influence the work reported in this paper.

 \nocite{*}
\bibliographystyle{abbrv}
\bibliography{references} 

@article{LQCG,
  title={Linear quotients of connected ideals of graphs},
  author={Ananthnarayan, H and Javadekar, Omkar and Maithani, Aryaman},
  journal={Journal of Algebraic Combinatorics},
  volume={61},
  number={3},
  pages={34},
  year={2025},
  publisher={Springer}
}

@article{RoyPatra2025,
  author    = {Amit Roy and Sourav Kanti Patra},
  title     = {Graded Betti Numbers of a Hyperedge Ideal Associated to Join of Graphs},
  journal   = {Bulletin of the Malaysian Mathematical Sciences Society},
  volume    = {48},
  pages     = {116},
  year      = {2025},
  doi       = {10.1007/s40840-025-01897-3},
}

@article{2025connectedideals,
      title={Stanley-{R}eisner ideals of higher independence complexes of chordal graphs.}, 
      author={Kanoy Kumar Das and Amit Roy and Kamalesh Saha},
      year={2026},
      journal={International Journal of Algebra and Computation, online ready},
      primaryClass={math.CO},
      url={https://doi.org/10.1142/S021819672650013X}, 
}

@article {EagonReiner1998,
    AUTHOR = {Eagon, John A. and Reiner, Victor},
     TITLE = {Resolutions of {S}tanley-{R}eisner rings and {A}lexander
              duality},
   JOURNAL = {J. Pure Appl. Algebra},
  FJOURNAL = {Journal of Pure and Applied Algebra},
    VOLUME = {130},
      YEAR = {1998},
    NUMBER = {3},
     PAGES = {265--275},
      ISSN = {0022-4049,1873-1376},
   MRCLASS = {13D25 (13D02 13H10)},
  MRNUMBER = {1633767},
MRREVIEWER = {Ralf\ Fr\"oberg},
       DOI = {10.1016/S0022-4049(97)00097-2},
       URL = {https://doi.org/10.1016/S0022-4049(97)00097-2},
}

@book {RHVBook,
    AUTHOR = {Villarreal, Rafael H.},
     TITLE = {Monomial algebras},
    SERIES = {Monographs and Research Notes in Mathematics},
   EDITION = {Second},
 PUBLISHER = {CRC Press, Boca Raton, FL},
      YEAR = {2015},
     PAGES = {xviii+686},
      ISBN = {978-1-4822-3469-5},
   MRCLASS = {13-02 (05C65 05E15 90C27)},
  MRNUMBER = {3362802},
MRREVIEWER = {Siamak\ Yassemi},
}

@article {HaWoodroofe2014,
    AUTHOR = {H\`a, Huy T\`ai and Woodroofe, Russ},
     TITLE = {Results on the regularity of square-free monomial ideals},
   JOURNAL = {Adv. in Appl. Math.},
  FJOURNAL = {Advances in Applied Mathematics},
    VOLUME = {58},
      YEAR = {2014},
     PAGES = {21--36},
      ISSN = {0196-8858,1090-2074},
   MRCLASS = {13D02 (05E40 13F20 55U10)},
  MRNUMBER = {3213742},
MRREVIEWER = {Benjamin\ P.\ Richert},
       DOI = {10.1016/j.aam.2014.05.002},
       URL = {https://doi.org/10.1016/j.aam.2014.05.002},
}

@article{DDIG,
  title={Distance r-domination number and r-independence complexes of graphs},
  author={Deshpande, Priyavrat and Shukla, Samir and Singh, Anurag},
  journal={European Journal of Combinatorics},
  volume={102},
  pages={103508},
  year={2022},
  publisher={Elsevier}
}

@article{CCVD,
  title={Chordality of clutters with vertex decomposable dual and ascent of clutters},
  author={Nikseresht, Ashkan},
  journal={Journal of Combinatorial Theory, Series A},
  volume={168},
  pages={318--337},
  year={2019},
  publisher={Elsevier}
}

@article{Hochster77,
  title={Cohen-{M}acaulay rings, combinatorics, and simplicial complexes},
  author={Hochster, Melvin},
  journal={Ring theory, II, Proc. Second Conf., Univ. Oklahoma, Norman, Okla., in: Lecture Notes Pure Appl. Math.},
  volume={26},
  pages={171--223},
  year={1977},
  publisher={Dekker, New York-Basel}
}

@article {ProBill1980,
    AUTHOR = {Provan, J. Scott and Billera, Louis J.},
     TITLE = {Decompositions of simplicial complexes related to diameters of
              convex polyhedra},
   JOURNAL = {Math. Oper. Res.},
  FJOURNAL = {Mathematics of Operations Research},
    VOLUME = {5},
      YEAR = {1980},
    NUMBER = {4},
     PAGES = {576--594},
      ISSN = {0364-765X,1526-5471},
   MRCLASS = {52A25 (90C05)},
  MRNUMBER = {593648},
MRREVIEWER = {J.\ Parida},
       DOI = {10.1287/moor.5.4.576},
       URL = {https://doi.org/10.1287/moor.5.4.576},
}

@article {HerzogHibiZheng2004,
    AUTHOR = {Herzog, J\"urgen and Hibi, Takayuki and Zheng, Xinxian},
     TITLE = {Dirac's theorem on chordal graphs and {A}lexander duality},
   JOURNAL = {European J. Combin.},
  FJOURNAL = {European Journal of Combinatorics},
    VOLUME = {25},
      YEAR = {2004},
    NUMBER = {7},
     PAGES = {949--960},
      ISSN = {0195-6698,1095-9971},
   MRCLASS = {05C38 (13F55 52B20)},
  MRNUMBER = {2083448},
MRREVIEWER = {Andrew\ Vince},
       DOI = {10.1016/j.ejc.2003.12.008},
       URL = {https://doi.org/10.1016/j.ejc.2003.12.008},
}

@article {MoradiAhang2016,
    AUTHOR = {Moradi, Somayeh and Khosh-Ahang, Fahimeh},
     TITLE = {On vertex decomposable simplicial complexes and their
              {A}lexander duals},
   JOURNAL = {Math. Scand.},
  FJOURNAL = {Mathematica Scandinavica},
    VOLUME = {118},
      YEAR = {2016},
    NUMBER = {1},
     PAGES = {43--56},
      ISSN = {0025-5521,1903-1807},
   MRCLASS = {13F55 (13D02)},
  MRNUMBER = {3475100},
MRREVIEWER = {Siamak\ Yassemi},
       DOI = {10.7146/math.scand.a-23295},
       URL = {https://doi.org/10.7146/math.scand.a-23295},
}

@article{SBCC,
  title={Stability of {B}etti numbers under reduction processes: Towards chordality of clutters},
  author={Bigdeli, Mina and Pour, Ali Akbar Yazdan and Zaare-Nahandi, Rashid},
  journal={Journal of Combinatorial Theory, Series A},
  volume={145},
  pages={129--149},
  year={2017}
}

@article{CHVD,
  title={Chordal graphs, higher independence and vertex decomposable complexes},
  author={Abdelmalek, Fred M and Deshpande, Priyavrat and Goyal, Shuchita and Roy, Amit and Singh, Anurag},
  journal={International Journal of Algebra and Computation},
  volume={33},
  number={03},
  pages={481--498},
  year={2023},
  publisher={World Scientific}
}

@article{FVDH,
  title={Fr{\"o}berg’s theorem, vertex splittability and higher independence complexes},
  author={Deshpande, Priyavrat and Roy, Amit and Singh, Anurag and Van Tuyl, Adam},
  journal={Journal of Commutative Algebra},
  volume={16},
  number={4},
  pages={391--410},
  year={2024},
  publisher={Rocky Mountain Mathematics Consortium Tempe, AZ, USA}
}

@article{TCG1,
author = {Bayer, Margaret and Denker, Mark and Jelić Milutinović, Marija and Rowlands, Rowan and Sundaram, Sheila and Xue, Lei},
year = {2024},
month = {05},
pages = {1630-1675},
title = {Topology of Cut Complexes of Graphs},
volume = {38},
journal = {SIAM Journal on Discrete Mathematics},
doi = {10.1137/23M1569034}
}

@article{TCG2,
author = {Bayer, Margaret and Denker, Mark and Jelić Milutinović, Marija and Sundaram, Sheila and Xue, Lei},
year = {2025},
month = {05},
pages = {1123-1157},
title = {Topology of Cut Complexes {II}},
volume = {39},
journal = {SIAM Journal on Discrete Mathematics},
doi = {10.1137/24M1676077}
}

@article{HICH,
author = {Deshpande, Priyavrat and Singh, Anurag},
year = {2021},
month = {03},
pages = {pp. 53–71},
title = {Higher independence complexes of graphs and their homotopy types},
volume = {36},
journal = {Journal of the Ramanujan Mathematical Society}
}

@incollection {OSR,
    AUTHOR = {Fr\"{o}berg, Ralf},
     TITLE = {On {S}tanley-{R}eisner rings},
 BOOKTITLE = {Topics in algebra, {P}art 2 ({W}arsaw, 1988)},
    SERIES = {Banach Center Publ.},
    VOLUME = {26},
     PAGES = {57--70},
 PUBLISHER = {PWN, Warsaw},
      YEAR = {1990},
   MRCLASS = {13D25 (13H10)},
  MRNUMBER = {1171260},
}

@article{CEI,
  title={Stanley-{R}eisner ideals with linear powers},
  author={Antonino Ficarra and Somayeh Moradi},
  journal={arXiv:2508.10354},
  year={2025}
}

@article{CSCM,
author = {Woodroofe, Russ},
year = {2009},
month = {11},
pages = {},
title = {Chordal and Sequentially {C}ohen-{M}acaulay Clutters},
volume = {18},
journal = {The electronic journal of combinatorics},
doi = {10.37236/695}
}

@article{Emtander2008,
author = {Emtander, Eric},
year = {2008},
month = {04},
pages = {},
title = {A class of hypergraphs that generalizes chordal graphs},
volume = {106},
journal = {Mathematica Scandinavica},
doi = {10.7146/math.scand.a-15124}
}

@article{WSHAG,
author = {Paolini, Giovanni and Salvetti, Mario},
year = {2017},
month = {03},
pages = {},
title = {Weighted sheaves and homology of {A}rtin groups},
volume = {18},
journal = {Algebraic \& Geometric Topology},
doi = {10.2140/agt.2018.18.3943}
}

@article{reisner76,
  title={Cohen-{M}acaulay quotients of polynomial rings},
  author={Reisner, Gerald Allen},
  journal={Advances in Mathematics},
  volume={21},
  number={1},
  pages={30--49},
  year={1976},
  publisher={Academic Press}
}

@article{tuyl_vill08,
title = {Shellable graphs and sequentially {C}ohen–{M}acaulay bipartite graphs},
journal = {Journal of Combinatorial Theory, Series A},
volume = {115},
number = {5},
pages = {799-814},
year = {2008},
issn = {0097-3165},
doi = {https://doi.org/10.1016/j.jcta.2007.11.001},
author = {Adam {Van Tuyl} and Rafael H. Villarreal}
}

@article {ST06,
    AUTHOR = {Szab\'{o}, Tibor and Tardos, G\'{a}bor},
     TITLE = {Extremal problems for transversals in graphs with bounded
              degree},
   JOURNAL = {Combinatorica},
  FJOURNAL = {Combinatorica. An International Journal on Combinatorics and
              the Theory of Computing},
    VOLUME = {26},
      YEAR = {2006},
    NUMBER = {3},
     PAGES = {333--351},
      ISSN = {0209-9683},
   MRCLASS = {05D15 (05C35)},
  MRNUMBER = {2246152},
MRREVIEWER = {Jen\H{o} Lehel},
       DOI = {10.1007/s00493-006-0019-9},
       URL = {https://doi.org/10.1007/s00493-006-0019-9},
}

@misc{ghoshSelva,
      title={Linear resolution of connected graph ideals and their powers}, 
      author={Arka Ghosh and S Selvaraja},
      year={arXiv:2512.06346},
      eprint={2512.06346},
      archivePrefix={arXiv},
      primaryClass={math.AC},
      url={https://arxiv.org/abs/2512.06346}, 
}

\end{document}